\documentclass[11pt]{article}

\usepackage{amsmath,amssymb,amsthm,mathrsfs}
\usepackage[margin=1in]{geometry}
\usepackage{hyperref}

\hypersetup{
  colorlinks=true,
  linkcolor=blue,
  citecolor=blue,
  urlcolor=blue
}

\newtheorem{theorem}{Theorem}[section]
\newtheorem{proposition}[theorem]{Proposition}
\newtheorem{lemma}[theorem]{Lemma}
\newtheorem{corollary}[theorem]{Corollary}
\theoremstyle{definition}
\newtheorem{definition}[theorem]{Definition}
\theoremstyle{remark}
\newtheorem{remark}[theorem]{Remark}
\theoremstyle{definition}
\newtheorem{example}[theorem]{Example}
\theoremstyle{plain}

\newcommand{\Hom}{\mathrm{Hom}}
\newcommand{\Irr}{\mathrm{Irr}}
\newcommand{\Z}{\mathrm{Z}}
\newcommand{\C}{\mathrm{C}}

\newcommand{\Prob}{\mathbf{P}}

\newcommand{\Comm}{\mathrm{Comm}}
\newcommand{\Aut}{\mathrm{Aut}}

\newcommand{\Fix}{\mathrm{Fix}}
\newcommand{\citedin}[1]{\newblock \emph{Cited in:} #1.}

\newcommand{\VecG}{\ensuremath{\mathrm{Vec}_{G}}}

\makeatletter
\def\@fnsymbol#1{%
  \ifcase#1\or
    \TextOrMath\textasteriskcentered \ast%
  \or
    \TextOrMath{\textasteriskcentered\textasteriskcentered}{\ast\ast}%
  \else
    \@ctrerr
  \fi
}

\makeatother

\newcommand{\Author}[2]{%
    {#1}\thanks{Email: \href{mailto:#2}{\texttt{#2}}}\\\Affiliation}

\newcommand{\Department}[1]{\def\DeptName{#1}}
\newcommand{\Faculty}[1]{\def\FacName{#1}}
\newcommand{\University}[1]{\def\UnivName{#1}}

\Department{Department of Mathematics}
\Faculty{Faculty of Natural Sciences}
\University{Ariel University, Ariel, Israel}

\author{
\Author{Vadim E. Levit}
{levitv@ariel.ac.il}
\and 
\Author{Robert Shwartz}
{robertsh@ariel.ac.il}
}

\title{Higher Commutativity in Finite Groups:\\[0.2em]
Rigidity, Extremal Bounds,\\
and Heisenberg-Type Families}

\date{}

\begin{document}
\maketitle

\begin{abstract}
For a finite group $G$ and an integer $r\ge 2$ let
\[
P_r(G):=\frac{|\Hom(\mathbb Z^r,G)|}{|G|^r},
\]
where $\Hom(\mathbb Z^r,G)$ is the set of pairwise commuting $r$-tuples in $G$. This paper studies rigidity and extremal behavior of the hierarchy $\{P_r(G)\}_{r\ge2}$, together with a low-rank representation-theoretic / TQFT counting bridge. The first main direction is cyclic-index rigidity: for groups with an abelian normal subgroup $A$ and cyclic quotient of order $\omega$, under a natural fixed-subgroup hypothesis we prove the exact all-rank formula
\[
P_r(G)=\frac{1}{\omega^r}+\left(1-\frac{1}{\omega^r}\right)\left(\frac{|A\cap Z(G)|}{|A|}\right)^{r-1},
\]
which yields gap and rigidity statements for non-abelian abelian extensions of prime index. The second main direction is the class-$2$ exponent-$p$ world. We develop a symplectic reduction, obtain closed formulas when $|G'|=p$, and prove a closed all-$r$ hierarchy in the $\mathbb F_q$-Heisenberg family:
\[
P_r(G)=q^{-2nr}\sum_{k=0}^{\min(n,r)}L_{n,k}(q)\prod_{i=0}^{k-1}(q^r-q^i).
\]
In particular, inside the $\mathbb F_q$-Heisenberg family the pair $(P_2(G),P_3(G))$ already determines the isoclinism class. Combining the cyclic-index formula with the known sharp upper bound for the multiple commutativity degree gives equality and near-extremal rigidity, including a stability gap near $11/32$ for commuting triples. At the low-rank end we also prove explicit class-number formulas for $P_3(G)$ and $P_4(G)$; these recover the simple-count formulas for the untwisted Drinfeld double and the untwisted quantum triple / double-loop-groupoid algebra.
\end{abstract}



\section{Introduction}

For a finite group $G$ and an integer $r\ge 2$ define
\[
P_r(G):=\frac{|\Hom(\mathbb Z^r,G)|}{|G|^r},
\qquad
\kappa_r(G):=\big|\Hom(\mathbb Z^r,G)/G\big|,
\]
where $G$ acts on $\Hom(\mathbb Z^r,G)$ by diagonal conjugation. Thus $P_2(G)=k(G)/|G|$, the classical commuting probability. This invariant goes back at least to Gustafson's theorem on the probability that two group elements commute and was developed structurally by Lescot and by Guralnick--Robinson \cite{Gustafson1973,Lescot1995,GuralnickRobinson2006}; in Lescot's notation $P_r(G)=d_{r-1}(G)$. The companion paper \cite{LevitShwartzPartI} develops the exact asymptotic and finite-spectrum theory of higher commutativity. The present paper has a different center of gravity: we study rigidity, extremal behavior, and exact low-rank formulas.

Higher commuting probabilities belong to a wider circle of probabilistic invariants of finite groups. Besides the classical commuting probability, recent work has also examined the probability that two random elements have the same centralizer and the probability that two random elements generate a nilpotent subgroup, again with structural consequences for the ambient group \cite{RahimiradZarrin2025,Lucchini2026}.

Generalized commuting probabilities attached to permutation equalities were studied by Cherniavsky et al. via Hultman numbers, and by Shwartz--Levit in the signed setting \cite{CherniavskyGoldsteinLevitShwartz2017,ShwartzLevit2022}. Our contribution in this landscape is to show that the full hierarchy $\{P_r(G)\}_{r\ge2}$ often behaves rigidly once one is near extremal values or inside natural structural families.

The paper has three main structural themes. The first is prime-index and cyclic-quotient rigidity. This direction is adjacent to ordinary commuting-degree computations for semidirect products of finite abelian groups, such as Nath's formula in a natural semidirect-product family \cite{Nath2013}. If $A\trianglelefteq G$ is abelian, $G/A\cong C_{\omega}$, and all nontrivial powers of a lift of a generator have the same fixed subgroup in $A$, then we prove the exact formula
\[
P_r(G)=\frac{1}{\omega^r}+\left(1-\frac{1}{\omega^r}\right)\left(\frac{|A\cap Z(G)|}{|A|}\right)^{r-1}.
\]
For prime index this yields immediate gap and rigidity statements, and on the extremal isoclinism class $G/Z(G)\cong C_p\times C_p$ it recovers the explicit hierarchy
\[
P_r(G)=\frac{p^r+p^{r-1}-1}{p^{2r-1}}.
\]

The second theme is extremal behavior near the top of the hierarchy. Here we recall, rather than claim as new, the known sharp upper bound for the multiple commutativity degree \cite[Prop.~2.4]{RezaeiNiroomandErfanian2014}. We include a compact proof only because the equality and deficit mechanisms are used in the rigidity statements that follow. In the form relevant here, if $p$ is the smallest prime divisor of $|G|$, then
\[
P_r(G)\le \frac{p^r+p^{r-1}-1}{p^{2r-1}},
\]
with equality if and only if $G/Z(G)\cong C_p\times C_p$. Using this known bound together with the prime-index formula, we derive a stability gap near $11/32$ for commuting triples and a $p$-group rigidity ladder that propagates near-extremal information through the hierarchy.

The third theme is the class-$2$ exponent-$p$ world. There the commutator induces an alternating form on $G/Z(G)$, so higher commuting probabilities can be studied via isotropic tuples in finite symplectic spaces. For $|G'|=p$ this yields a recursion and closed formulas for low ranks, together with a rigidity theorem saying that the full hierarchy determines isoclinism. More generally, for groups of $\mathbb F_q$-Heisenberg type of rank $n$ we prove the closed all-$r$ formula
\[
P_r(G)=q^{-2nr}\sum_{k=0}^{\min(n,r)}L_{n,k}(q)\prod_{i=0}^{k-1}(q^r-q^i),
\]
where $L_{n,k}(q)$ counts totally isotropic $k$-subspaces of $\mathbb F_q^{2n}$. In particular,
\[
m(G)=q^{n+1},
\qquad
N_{\max}(G)=\prod_{i=1}^{n}(q^i+1),
\]
and already $P_2(G)$ and $P_3(G)$ determine the isoclinism class inside this family.

A complementary low-rank theme runs through the paper. We prove exact class-number formulas for $P_3(G)$ and $P_4(G)$; in particular,
\[
|G|^2P_3(G)=|\Irr(D(G))|,
\qquad
|G|^3P_4(G)=|\Irr(\mathbb C[\Lambda^2G])|.
\]
These identities give a representation-theoretic counting bridge between higher commutativity and the untwisted finite-gauge / Dijkgraaf--Witten side of topological field theory \cite{DijkgraafWitten1990,FreedQuinn1993,SchroederTongViet2025}. They are used here as counting interpretations rather than as new tensor-categorical constructions, but they help place the low-rank invariants in a broader algebraic-topological context; recent work of Schroeder and Tong--Viet explicitly treats finite-group invariants from Dijkgraaf--Witten theory as higher-genus analogues of commuting probability \cite{SchroederTongViet2025}.

In summary, Part~I treats higher commutativity as an asymptotic/spectral invariant, whereas Part~II treats it as a rigidity/extremal invariant. Sections~2 and~3 record only the standing notation and the recursion facts imported from the companion paper \cite{LevitShwartzPartI}. Section~4 develops the low-rank class-number formulas, the loop-groupoid counting interpretations, and the prime-index rigidity statements. Sections~5 and~6 organize the extremal bounds and the $11/32$ stability gap. Section~7 develops the class-$2$ exponent-$p$ and Heisenberg-family theory, and Section~8 concludes.
\section{Standing notation and isoclinism}

For a finite group $G$ and an integer $r\ge 1$ define
\[
\Comm_r(G):=\{(x_1,\dots,x_r)\in G^r : [x_i,x_j]=1\ \forall i,j\},
\qquad
P_r(G):=\frac{|\Comm_r(G)|}{|G|^r}.
\]
Equivalently, $\Comm_r(G)=\Hom(\mathbb Z^r,G)$ and $P_r(G)=|\Hom(\mathbb Z^r,G)|/|G|^r$.
As in Lescot's notation, $P_r(G)=d_{r-1}(G)$ for every $r\ge 2$ \cite{Lescot1995,RezaeiNiroomandErfanian2014}.
The companion paper \cite{LevitShwartzPartI} develops the full asymptotic/spectral picture and proves several basic formal properties of the hierarchy.  Here we retain only the pieces that are used explicitly below.

\subsection{Isoclinism and invariance of \texorpdfstring{$P_r$}{Pr}}\label{subsec:isoclinism}

It is natural to regard the family $\{P_r(G)\}_{r\ge 2}$ as a ``commutator-geometry'' invariant of $G$.
This is made precise by \emph{isoclinism}, introduced by P.~Hall \cite{Hall1940} and used systematically by
Lescot in the study of commuting probabilities \cite{Lescot1995}.

\begin{definition}[Isoclinism \cite{Hall1940}]\label{def:isoclinism}
Two groups $G,H$ are \emph{isoclinic} if there exist isomorphisms
\[
\phi:G/\Z(G)\xrightarrow{\ \sim\ } H/\Z(H),
\qquad
\psi:G'\xrightarrow{\ \sim\ } H',
\]
such that the commutator maps are intertwined:
\[
\psi\big([g_1,g_2]\big)=\big[\widehat{\phi}(g_1\Z(G)),\,\widehat{\phi}(g_2\Z(G))\big]
\quad\text{for all }g_1,g_2\in G,
\]
where $\widehat{\phi}(g\Z(G))$ denotes any lift of $\phi(g\Z(G))$ to $H$.
Equivalently, if we write
\[
\kappa_G:(G/\Z(G))^2\to G',\qquad (g\Z(G),h\Z(G))\mapsto [g,h],
\]
(and similarly $\kappa_H$), then $\psi\circ \kappa_G=\kappa_H\circ(\phi\times\phi)$.
\end{definition}

\begin{proposition}[Isoclinism invariance of higher commuting probabilities]\label{prop:isoclinism-invariant}
If $G$ and $H$ are isoclinic finite groups, then
\[
P_r(G)=P_r(H)\qquad\text{for all }r\ge 2.
\]
In particular, the entire sequence $\{P_r(G)\}_{r\ge 2}$ depends only on the isoclinism class of $G$.
\end{proposition}

\begin{proof}
See \cite[\S2.1]{LevitShwartzPartI}.
\end{proof}

\section{Centralizer recursion and orbit counts}

\subsection{A recursion via centralizers}

\begin{lemma}[Centralizer recursion]\label{lem:recursion}
For every finite group $G$ and every integer $r\ge 2$,
\[
P_r(G)=\frac{1}{|G|^r}\sum_{x\in G}|\C_G(x)|^{r-1}P_{r-1}(\C_G(x)).
\]
Equivalently, grouping the preceding sum by conjugacy classes gives
\[
|G|^rP_r(G)=|G|\sum_{[g]\subseteq G}|\C_G(g)|^{r-2}P_{r-1}(\C_G(g)),
\]
where the sum runs over conjugacy classes in $G$.
\end{lemma}

\begin{proof}
See \cite[Lemma~3.1]{LevitShwartzPartI}.
\end{proof}

\subsection{Orbit-count normalization and integrality}

\begin{definition}[Higher class numbers]\label{def:kappa}
For $r\ge 0$ define
\[
\kappa_r(G):=\big|\Hom(\mathbb Z^r,G)/G\big|,
\]
where $G$ acts by diagonal conjugation on $\Hom(\mathbb Z^r,G)$.
By convention, $\kappa_0(G)=1$ and $\kappa_1(G)=k(G)$.
\end{definition}

\begin{theorem}[Burnside orbit lemma normalization]\label{thm:burnside}
For every finite group $G$ and every $r\ge 1$,
\[
\kappa_r(G)=|G|^r\,P_{r+1}(G)
\qquad\text{equivalently}\qquad
P_{r+1}(G)=\frac{\kappa_r(G)}{|G|^r}.
\]
In particular, $|G|^rP_{r+1}(G)$ is always an integer.
\end{theorem}

\begin{proof}
See \cite[Theorem~3.4]{LevitShwartzPartI}.
\end{proof}

\section{Low-rank formulas and prime-index rigidity}

\subsection{An explicit formula for \texorpdfstring{$P_3(G)$}{P3(G)}}
\begin{theorem}[Centralizer class-number formula]\label{thm:P3formula}
For every finite group $G$,
\[
P_3(G)=\frac{1}{|G|^3}\sum_{x\in G}|\C_G(x)|\,k(\C_G(x))
=\frac{1}{|G|^2}\sum_{[g]\subseteq G} k(\C_G(g)).
\]
Equivalently,
\[
|G|^2P_3(G)=\kappa_2(G)=\sum_{[g]}k(\C_G(g)).
\]
\end{theorem}

\begin{proof}
By Lemma \ref{lem:recursion} for $r=3$,  $P_3(G)=\frac{1}{|G|^3}\sum_{x\in G}|C_G(x)|^2 P_2(C_G(x))$.
Set $C:=\C_G(x)$, by using $P_2(H)=k(H)/|H|$, we obtain
\[
\begin{aligned}
|C|^2 P_2(C)
&= |C|^2\cdot \frac{k(C)}{|C|} \\
&= |C|\,k(C).
\end{aligned}
\]
Summing over $x$ yields the first identity.

For the second identity, group by conjugacy classes: if $[g]$ is a conjugacy class of size
$|[g]|$, then $|\C_G(g)|=|G|/|[g]|$, and $k(\C_G(x))$ is constant for $x\in[g]$.
Hence the contribution of the class $[g]$ to $\sum_{x\in G}|\C_G(x)|\,k(\C_G(x))$
equals
\[
|[g]|\cdot \frac{|G|}{|[g]|}\,k(\C_G(g))=|G|\,k(\C_G(g)).
\]
Divide by $|G|^3$ to obtain $P_3(G)=(1/|G|^2)\sum_{[g]}k(\C_G(g))$.
The final claim is Theorem~\ref{thm:burnside} for $r=2$.
\end{proof}

\begin{corollary}\label{dihedral-P3-lim}
Let $G=D_{2n}$ be the dihedral group of order $2n$. Then
\begin{itemize}
\item If $n$ is odd, then $P_3(G)=\frac{1}{8}+\frac{7}{8n^2}$.
\item If $n$ is even, then $P_3(G)=\frac{1}{8}+\frac{7}{2n^2}$.
\item In particular, $\lim_{n\to\infty}P_3(D_{2n})=\frac{1}{8}$.
\end{itemize}
\end{corollary}

\begin{proof}
Write
\[
D_{2n}=\langle a,b \mid a^n=b^2=1,\; bab=a^{-1}\rangle.
\]

\noindent\emph{Case 1: $n$ odd.}
The conjugacy classes are:
\begin{itemize}
\item $\{1\}$,
\item $\{a^k,a^{n-k}\}$ for $1\le k\le (n-1)/2$,
\item the class of reflections $\{a^k b:0\le k\le n-1\}$.
\end{itemize}
Moreover:
\begin{itemize}
\item For $1\le k\le n-1$, $\C_G(a^k)=\langle a\rangle\cong C_n$ is abelian of order $n$, so $k(\C_G(a^k))=n$.
\item For each reflection $a^k b$, $\C_G(a^k b)$ has order $2$ (hence is abelian), so $k(\C_G(a^k b))=2$.
\item $\C_G(1)=G$, and $k(G)=\frac{n-1}{2}+2=\frac{n+3}{2}$.
\end{itemize}
Therefore, by Theorem~\ref{thm:P3formula},
\[
P_3(D_{2n})
=\frac{1}{(2n)^2}\sum_{[g]\subseteq D_{2n}} k(\C_{D_{2n}}(g))
=\frac{1}{4n^2}\left( \frac{n(n-1)}{2}+2+\frac{n+3}{2}\right)
=\frac{n^2+7}{8n^2}
=\frac18+\frac{7}{8n^2}.
\]

\noindent\emph{Case 2: $n$ even.}
The conjugacy classes are:
\begin{itemize}
\item $\{1\}$ and $\{a^{n/2}\}$ (both central),
\item $\{a^k,a^{n-k}\}$ for $1\le k\le (n-2)/2$,
\item the two reflection classes $\{a^{2k}b:0\le k\le (n-2)/2\}$ and $\{a^{2k+1}b:0\le k\le (n-2)/2\}$.
\end{itemize}
Moreover:
\begin{itemize}
\item For $1\le k\le n-1$ with $k\ne n/2$, $\C_G(a^k)=\langle a\rangle$ is abelian of order $n$, so $k(\C_G(a^k))=n$.
\item For each reflection $a^k b$, $\C_G(a^k b)$ has order $4$ (hence is abelian), so $k(\C_G(a^k b))=4$.
\item $\C_G(1)=\C_G(a^{n/2})=G$, and $k(G)=\frac{n}{2}+3$.
\end{itemize}
Thus Theorem~\ref{thm:P3formula} yields
\[
P_3(D_{2n})
=\frac{1}{4n^2}\left( \frac{n(n-2)}{2}+8+2\left(\frac{n}{2}+3\right)\right)
=\frac{n^2+28}{8n^2}
=\frac18+\frac{7}{2n^2}.
\]
The limit statement is immediate.
\end{proof}

\subsection{The Drinfeld double}\label{subsec:double}

We briefly record the standard low-rank bridge between commuting triples and the untwisted Drinfeld double.
Throughout this subsection we work over $\mathbb C$.

\begin{definition}[Untwisted Drinfeld double $D(G)$]\label{def:drinfeld-double}
Let $G$ be a finite group and let $\mathbb C^G$ be the algebra of functions $G\to \mathbb C$.
The \emph{(untwisted) Drinfeld double} $D(G)$ \cite{DrinfeldICM,DPR,EGNO} is the semidirect (crossed) product Hopf algebra
\[
D(G):=\mathbb C^G \rtimes \mathbb C G,
\]
whose underlying vector space is $\mathbb C^G\otimes \mathbb C G$ and whose multiplication is
determined on the basis elements $\{\delta_g\otimes x: g,x\in G\}$ by
\[
(\delta_g\otimes x)\,(\delta_h\otimes y)=
\delta_{g,\,x h x^{-1}}\;(\delta_g\otimes xy).
\]
It is a finite-dimensional semisimple quasitriangular Hopf algebra whose representation category is braided
and is equivalent to the Drinfeld center $\mathcal Z(\VecG)$ \cite{DPR,EGNO}.
\end{definition}

\begin{remark}[How simples of $D(G)$ are labeled]\label{rem:double-simples}
A standard feature of $D(G)$ is the parametrization of its simple modules by pairs $([g],\rho)$ where
$[g]$ is a conjugacy class in $G$ and $\rho\in\Irr(\C_G(g))$ is an irreducible representation of the centralizer.
Concretely, the commutative semisimple subalgebra $\mathbb C^G$ has primitive idempotents $\delta_g$, so a finite-dimensional
$D(G)$-module decomposes according to conjugacy support, and the stabilizer of a homogeneous piece over $g$ is $\C_G(g)$.
See, for instance, \cite{DPR,Witherspoon,EGNO}.
\end{remark}

\begin{theorem}[Commuting triples and the quantum double]\label{thm:double}
For every finite group $G$,
\[
|\Irr(D(G))|
=\sum_{[g]\subseteq G} k(\C_G(g))
=|G|^2P_3(G).
\]
\end{theorem}

\begin{proof}
By Remark~\ref{rem:double-simples}, for each conjugacy class $[g]$ there are exactly $k(\C_G(g))$ simple
$D(G)$-modules lying over $[g]$, hence
\[
|\Irr(D(G))|=\sum_{[g]} k(\C_G(g)).
\]
The equality with $|G|^2P_3(G)$ is exactly Theorem~\ref{thm:P3formula}.
\end{proof}

\begin{remark}[A loop-groupoid viewpoint]\label{rem:double-as-loopgroupoid}
Let $\Lambda G$ be the loop groupoid of $G$, whose objects are elements of $G$ and whose morphisms are conjugations.
Up to the usual opposite-convention harmlessness, the groupoid algebra $\mathbb C[\Lambda G]$ identifies with the crossed product
$\mathbb C^G\rtimes \mathbb C G$ from Definition~\ref{def:drinfeld-double}; compare Willerton's loop-groupoid description of the twisted Drinfeld double \cite{Willerton2008}. In this sense, Theorem~\ref{thm:double} is the first low-rank case of the general iterated loop-groupoid identity proved later in Theorem~\ref{thm:iterated-loop}.
\end{remark}

\subsection{Commuting quadruples and the quantum triple}

In this subsection we use only the resulting class-number and simple-count identities; no tensor-categorical structure of the quantum triple is constructed or needed.
We now push the commuting-triple dictionary one step further, from $P_3(G)$ to $P_4(G)$. We use the term
``quantum triple'' in the sense common in the Dijkgraaf--Witten / topological-phases literature \cite{BullivantDelcamp2019,BullivantDelcampCrossCircle}.

\begin{theorem}[Centralizer class-number formula for commuting quadruples]\label{thm:P4formula}
For every finite group $G$,
\[
P_4(G)=\frac{1}{|G|^3}\sum_{[g]\subseteq G}~\sum_{[h]\subseteq \C_G(g)} k(\C_G(g,h)),
\]
equivalently
\[
|G|^3P_4(G)=\kappa_3(G)=\sum_{[g]\subseteq G}~\sum_{[h]\subseteq \C_G(g)} k(\C_G(g,h)),
\]
where, for each $G$-conjugacy class $[g]$, the inner sum runs over the $\C_G(g)$-conjugacy classes in $\C_G(g)$; equivalently, the total sum is over $G$-conjugacy orbits of commuting pairs $(g,h)$.
\end{theorem}

\begin{proof}
Apply Lemma~\ref{lem:recursion} with $r=4$:
\[
|G|^4P_4(G)=\sum_{x\in G} |\C_G(x)|^3\,P_3(\C_G(x)).
\]
Fix $x\in G$.
Theorem~\ref{thm:P3formula} applied to the group $\C_G(x)$ gives
\[
|\C_G(x)|^3P_3(\C_G(x))=\sum_{y\in \C_G(x)} |\C_{\C_G(x)}(y)|\,k(\C_{\C_G(x)}(y)).
\]
But $\C_{\C_G(x)}(y)=\C_G(x,y)$, the common centralizer of $x$ and $y$ in $G$.
Summing over $x$ yields
\[
|G|^4P_4(G)=\sum_{\substack{(x,y)\in G^2\\ xy=yx}} |\C_G(x,y)|\,k(\C_G(x,y)).
\]
Now group the last sum by conjugacy orbits of commuting pairs.
If $(g,h)$ is a representative, then its orbit has size $|G|/|\C_G(g,h)|$, so its contribution is
\[
\frac{|G|}{|\C_G(g,h)|}\cdot |\C_G(g,h)|\,k(\C_G(g,h))=|G|\,k(\C_G(g,h)).
\]
Therefore
\[
|G|^4P_4(G)=|G|\sum_{[g]\subseteq G}~\sum_{[h]\subseteq \C_G(g)} k(\C_G(g,h)),
\]
and dividing by $|G|^4$ gives the claim.
\end{proof}

\paragraph{The double loop groupoid.}
Write $BG$ for the one-object groupoid with automorphism group $G$.
The double loop groupoid $\Lambda^2G:=\mathrm{Fun}(B\mathbb Z^2,BG)$ has objects the commuting pairs $(g,h)\in G^2$ and morphisms given by simultaneous conjugation.
The automorphism group at an object $(g,h)$ is the common centralizer $\C_G(g,h)$.

\begin{lemma}[Finite groupoid algebra decomposition]\label{lem:groupoid-decomp}
Let $\mathcal G$ be a finite groupoid. For each isomorphism class of objects $[x]\in \mathrm{Ob}(\mathcal G)/\cong$
choose a representative $x$.
Then there is a (noncanonical) $\mathbb C$-algebra isomorphism
\[
\mathbb C[\mathcal G]\ \cong\ \bigoplus_{[x]\in \mathrm{Ob}(\mathcal G)/\cong} M_{|[x]|}\big(\mathbb C[\Aut(x)]\big),
\]
where $|[x]|$ denotes the number of objects in the isomorphism class $[x]$.
Consequently,
\[
|\Irr(\mathbb C[\mathcal G])|
=\sum_{[x]\in \mathrm{Ob}(\mathcal G)/\cong} k(\Aut(x)).
\]
\end{lemma}

\begin{proof}
Since $\mathcal G$ is a groupoid, its connected components are exactly its isomorphism classes of objects.
So it suffices to treat the case when $\mathcal G$ is connected, i.e.\ $\mathrm{Ob}(\mathcal G)=[x]$ for some $x$.

Fix a base object $x$.
For each object $y\in [x]$ choose an isomorphism $\tau_y:x\to y$ with $\tau_x=\mathrm{id}_x$.
Let $n:=|[x]|$.
Define a linear map
\[
\Theta:\mathbb C[\mathcal G]\to M_n(\mathbb C[\Aut(x)])
\]
on the basis $\mathrm{Mor}(\mathcal G)$ by
\[
\Theta(f:y\to z)\ :=\ E_{z,y}\otimes\big(\tau_z^{-1}\circ f\circ \tau_y\big),
\]
where $E_{z,y}$ is the $(z,y)$ matrix unit indexed by the object set $[x]$, and
$\tau_z^{-1}\circ f\circ \tau_y\in\Aut(x)$.
A direct check shows that $\Theta$ is multiplicative:
if $f:y\to z$ and $g:z\to w$ are composable then
\[
\Theta(g)\Theta(f)
=(E_{w,z}\otimes \tau_w^{-1}g\tau_z)(E_{z,y}\otimes \tau_z^{-1}f\tau_y)
=E_{w,y}\otimes (\tau_w^{-1}g\tau_z)(\tau_z^{-1}f\tau_y)
=\Theta(g\circ f),
\]
and if $g$ and $f$ are not composable then both products are $0$.
Moreover, $\Theta$ is a bijection on bases: given $E_{z,y}\otimes a$ with $a\in\Aut(x)$, the inverse sends it to
the morphism $\tau_z\circ a\circ \tau_y^{-1}:y\to z$.
Thus $\Theta$ is an algebra isomorphism in the connected case, and taking the direct sum over components
yields the stated decomposition in general.

The formula for $|\Irr(\mathbb C[\mathcal G])|$ follows because $M_n(A)$ and $A$ have equivalent module categories,
so the number of simple modules of $M_n(\mathbb C[\Aut(x)])$ is $|\Irr(\mathbb C[\Aut(x)])|=k(\Aut(x))$.
\end{proof}

\begin{corollary}[Quantum triple count]\label{cor:quantumtriple}
For every finite group $G$,
\[
|\Irr(\mathbb C[\Lambda^2G])|=\kappa_3(G)=|G|^3P_4(G).
\]
Equivalently,
\[
P_4(G)=\frac{|\Irr(\mathbb C[\Lambda^2G])|}{|G|^3}.
\]
\end{corollary}

\begin{proof}
By Lemma~\ref{lem:groupoid-decomp}, the simple modules of $\mathbb C[\Lambda^2G]$ are counted by the sum of
$k(\C_G(g,h))$ over conjugacy-orbit representatives of commuting pairs $(g,h)$.
That is,
\[
|\Irr(\mathbb C[\Lambda^2G])|=\sum_{[g]\subseteq G}~\sum_{[h]\subseteq \C_G(g)} k(\C_G(g,h)).
\]
Theorem~\ref{thm:P4formula} identifies this sum with $|G|^3P_4(G)$.
\end{proof}

\begin{remark}[Topological meaning]\label{rem:crosscircle}
In untwisted Dijkgraaf--Witten theory one has
\[
Z_{\mathrm{DW}}(T^4)=\frac{|\Hom(\mathbb Z^4,G)|}{|G|}
=|G|^3P_4(G)
\]
\cite{DijkgraafWitten1990,FreedQuinn1993}. This viewpoint is compatible with the recent finite-group invariant program of Schroeder and Tong--Viet, where Dijkgraaf--Witten surface invariants are treated as higher-genus analogues of commuting probability \cite{SchroederTongViet2025}. Bullivant and Delcamp identify the same quantity with the simple-module count of the double loop-groupoid algebra, more generally for circle compactifications of surfaces \cite{BullivantDelcampCrossCircle}. Thus Corollary~\ref{cor:quantumtriple} is the untwisted $4$-torus counting formula in this language.
\end{remark}

\begin{theorem}[Iterated loop groupoids and higher commuting probabilities]\label{thm:iterated-loop}
Let $G$ be a finite group and let $r\ge 2$.
Consider the $(r{-}2)$-fold loop groupoid
\[
\Lambda^{r-2}G := \mathrm{Fun}(B\mathbb Z^{r-2}, BG),
\]
whose objects are commuting $(r{-}2)$-tuples $\mathbf g=(g_1,\dots,g_{r-2})\in \Comm_{r-2}(G)$ and whose morphisms are simultaneous conjugations by $G$.
Write
\[
\C_G(\mathbf g):=\bigcap_{i=1}^{r-2}\C_G(g_i)
\]
for the joint centralizer, with the convention that $\Comm_0(G)$ is a singleton and $\C_G(\varnothing)=G$ when $r=2$.
Then
\[
|\Irr(\mathbb C[\Lambda^{r-2}G])|
=\sum_{[\mathbf g]\in \Comm_{r-2}(G)/G} k(\C_G(\mathbf g))
=\kappa_{r-1}(G)
=|G|^{r-1}P_r(G).
\]
\end{theorem}

\begin{proof}
Objects of $\Lambda^{r-2}G$ are commuting $(r{-}2)$-tuples, and two such objects are isomorphic if and only if they are simultaneously conjugate by an element of $G$.
The isotropy group of $\mathbf g$ is its stabilizer for conjugation, namely the joint centralizer $\C_G(\mathbf g)$.
Therefore Lemma~\ref{lem:groupoid-decomp} gives
\[
|\Irr(\mathbb C[\Lambda^{r-2}G])|
=\sum_{[\mathbf g]\in \Comm_{r-2}(G)/G} k(\C_G(\mathbf g)).
\]

To identify the same sum with $\kappa_{r-1}(G)$, consider the $G$-equivariant projection
\[
\pi:\Comm_{r-1}(G)\to \Comm_{r-2}(G),\qquad (g_1,\dots,g_{r-2},t)\mapsto (g_1,\dots,g_{r-2}).
\]
Fix a $G$-orbit $[\mathbf g]\in \Comm_{r-2}(G)/G$ and choose a representative $\mathbf g$.
The fiber $\pi^{-1}(\mathbf g)$ consists of all $t\in G$ commuting with each $g_i$, namely $\pi^{-1}(\mathbf g)=\C_G(\mathbf g)$.
Moreover, the stabilizer of $\mathbf g$ is $\C_G(\mathbf g)$ and it acts on the fiber by conjugation.
Two points $(\mathbf g,t)$ and $(\mathbf g,t')$ in the fiber define $G$-conjugate commuting $(r{-}1)$-tuples if and only if $t$ and $t'$ are conjugate by an element of $\C_G(\mathbf g)$.
Hence the $G$-orbits in $\Comm_{r-1}(G)$ lying above $[\mathbf g]$ are in bijection with the conjugacy classes of $\C_G(\mathbf g)$, and their number is $k(\C_G(\mathbf g))$.
Summing over all $[\mathbf g]$ yields
\[
\kappa_{r-1}(G)=\sum_{[\mathbf g]\in \Comm_{r-2}(G)/G} k(\C_G(\mathbf g)).
\]
Finally, Theorem~\ref{thm:burnside} gives $\kappa_{r-1}(G)=|G|^{r-1}P_r(G)$.
\end{proof}

\begin{remark}
For $r=3$, Theorem~\ref{thm:iterated-loop} recovers Theorem~\ref{thm:double}; for $r=4$, it recovers Corollary~\ref{cor:quantumtriple}.
\end{remark}

Corollary \ref{dihedral-P3-lim} is a special case of the following exact formula, which in fact holds for all higher commutativity probabilities.

\begin{theorem}[Normal abelian cyclic index-$\omega$ formula]\label{prop:Pk-normal-index-p}
Let $G$ be a finite group and let $A\trianglelefteq G$ be an abelian normal subgroup such that $G/A$ is cyclic of order $\omega$.
Choose $t\in G$ whose image generates $G/A$, and set
\[
F:=\C_A(t)=A\cap \Z(G),\qquad f:=|F|,\qquad n:=|A|.
\]
Assume that
\[
\C_A(t^j)=F\qquad (1\le j\le \omega-1).
\]
Then for every $r\ge 1$,
\begin{equation}\label{eq:Comm-k-index-p}
|\Comm_r(G)|=n^r+(\omega^r-1)\,n\,f^{\,r-1}.
\end{equation}
Equivalently, for every $r\ge 1$,
\begin{equation}\label{eq:Pk-index-p}
P_r(G)
=\frac{1}{\omega^r}+\left(1-\frac{1}{\omega^r}\right)\left(\frac{f}{n}\right)^{r-1}.
\end{equation}
In particular, for $r=3$ one has
\[
P_3(G)=\frac{1}{\omega^3}+\left(1-\frac{1}{\omega^3}\right)\left(\frac{f}{n}\right)^2 .
\]
\end{theorem}

\begin{proof}
Fix $t\in G$ as in the statement, and let $\varphi\in \Aut(A)$ be conjugation by $t$.
Since $G/A\cong C_{\omega}$, the automorphism $\varphi$ has order dividing $\omega$.
Because $A$ is abelian and $G=\langle A,t\rangle$, we have
\[
F=\C_A(t)=A\cap \Z(G).
\]

Every element of $G$ is uniquely of the form $a t^\ell$ with $a\in A$ and $0\le \ell\le \omega-1$.
For $j\ge 0$ define
\[
\delta_j:A\to A,\qquad \delta_j(a):=\varphi^j(a)a^{-1}.
\]
Thus $\delta_0=1$, and for $1\le j\le \omega-1$ one has
\[
\ker(\delta_j)=\C_A(t^j)=F
\]
by hypothesis, so
\[
|\operatorname{im}(\delta_j)|=\frac{n}{f}.
\]
Moreover, for $j\ge 1$,
\[
\delta_j(a)=\delta_1(\varphi^{j-1}(a))\cdots \delta_1(\varphi(a))\,\delta_1(a),
\]
hence $\operatorname{im}(\delta_j)\subseteq \operatorname{im}(\delta_1)$.
Since both images have size $n/f$, it follows that
\[
\operatorname{im}(\delta_j)=\operatorname{im}(\delta_1)\qquad (1\le j\le \omega-1).
\]

We next record the centralizers needed for the counting argument.

If $a\in F$, then $a\in \Z(G)$ and $\C_G(a)=G$.
If $a\in A\setminus F$ and $y=b t^\ell\in G$ commutes with $a$, then
\[
a=yay^{-1}=t^\ell a t^{-\ell}=\varphi^\ell(a).
\]
If $\ell\ne 0$, this would force $a\in \C_A(t^\ell)=F$, a contradiction.
Hence $\ell=0$, so $y\in A$ and therefore
\[
\C_G(a)=A \qquad (a\in A\setminus F).
\]

Now fix $x=a t^k\in G\setminus A$, with $1\le k\le \omega-1$.
For $y=b t^\ell\in G$, one computes
\[
yx=b\,\varphi^\ell(a)\,t^{\ell+k},
\qquad
xy=a\,\varphi^k(b)\,t^{k+\ell}.
\]
Thus $y$ commutes with $x$ if and only if
\[
b\,\varphi^\ell(a)=a\,\varphi^k(b),
\]
equivalently,
\[
\delta_k(b)=\delta_\ell(a).
\]
For $\ell=0$ the right-hand side is $1$, so there are exactly $|\ker(\delta_k)|=f$ solutions $b$.
For $1\le \ell\le \omega-1$, the element $\delta_\ell(a)$ lies in $\operatorname{im}(\delta_\ell)=\operatorname{im}(\delta_k)$, so again there are exactly $|\ker(\delta_k)|=f$ solutions $b$.
Therefore each coset $At^\ell$ contributes exactly $f$ elements to $\C_G(x)$, and hence
\[
|\C_G(x)|=\omega f \qquad (x\in G\setminus A).
\]

Also,
\[
\C_G(x)\cap A=\C_A(x)=\C_A(t^k)=F,
\]
so the image of $\C_G(x)$ in $G/A$ has order
\[
\frac{|\C_G(x)|}{|\C_G(x)\cap A|}=\frac{\omega f}{f}=\omega.
\]
Hence $\C_G(x)/F\cong G/A\cong C_\omega$ is cyclic. Since $F\le \Z(G)$, it follows that $\C_G(x)$ is abelian.

Let $N_r:=|\Comm_r(G)|$.
Using the partition $G=F\sqcup (A\setminus F)\sqcup (G\setminus A)$, the centralizer descriptions above, and Lemma~\ref{lem:recursion}, we obtain
\begin{equation}\label{eq:Nk-recurrence}
N_r = f\,N_{r-1} + (n-f)\,n^{r-1} + (\omega-1)n\,(\omega f)^{r-1}.
\end{equation}

We claim that \eqref{eq:Comm-k-index-p} holds for all $r\ge 1$.
For $r=1$ one has
\[
N_1=|G|=\omega n=n+(\omega-1)n,
\]
which matches \eqref{eq:Comm-k-index-p}.
Assume \eqref{eq:Comm-k-index-p} holds for $r-1$.
Substituting
\[
N_{r-1}=n^{r-1}+(\omega^{r-1}-1)n f^{r-2}
\]
into \eqref{eq:Nk-recurrence} gives
\begin{align*}
N_r
&= f\bigl(n^{r-1}+(\omega^{r-1}-1)n f^{r-2}\bigr)+(n-f)n^{r-1}+(\omega-1)n(\omega f)^{r-1} \\
&= n^r+\bigl((\omega^{r-1}-1)+(\omega-1)\omega^{r-1}\bigr)n f^{r-1} \\
&= n^r+(\omega^r-1)n f^{r-1},
\end{align*}
proving \eqref{eq:Comm-k-index-p}.
Dividing by $|G|^r=(\omega n)^r$ yields \eqref{eq:Pk-index-p}.
\end{proof}

\begin{example}[Why the fixed-subgroup hypothesis is needed]\label{ex:cyclic-index-omega-nonfaithful}
Let
\[
G=\langle a,t \mid a^5=t^4=1,\ t a t^{-1}=a^{-1}\rangle,
\qquad
A:=\langle a\rangle\cong C_5.
\]
Then $A\trianglelefteq G$ is abelian and $G/A\cong C_4$ is cyclic.
Moreover
\[
F=\C_A(t)=1,
\qquad\text{but}\qquad
\C_A(t^2)=A,
\]
because $t^2$ acts trivially on $A$.
So the hypothesis $\C_A(t^j)=F$ for all $1\le j\le \omega-1$ fails.

This failure is genuine: a direct count gives
\[
P_2(G)=\frac{2}{5},
\]
whereas the formula of Theorem~\ref{prop:Pk-normal-index-p} with $\omega=4$, $n=5$, and $f=1$
would incorrectly predict
\[
\frac{1}{4^2}+\left(1-\frac{1}{4^2}\right)\frac{1}{5}=\frac14.
\]
Thus the equal-fixed-subgroup hypothesis in Theorem~\ref{prop:Pk-normal-index-p} cannot be omitted.
\end{example}

\begin{corollary}[metacyclic groups of order $p(p-1)$]\label{pp-1 groups}
Let $G$ be a non-abelian group of order $p(p-1)$ for a prime $p$, and let $A$ be its unique subgroup of order $p$.
Assume that $G/A$ is cyclic and that every element of $G\setminus A$ is noncentral.
Then for every $r\ge 2$,
\[
P_r(G)=\frac{1}{(p-1)^r}+\left(1-\frac{1}{(p-1)^r}\right)\frac{1}{p^{\,r-1}}.
\]
In particular,
\[
P_2(G)=\frac{1}{p-1}.
\]
\end{corollary}
\begin{proof}
By Cauchy's theorem, $G$ contains a subgroup $A$ of order $p$.
If $A$ and $B$ were distinct subgroups of order $p$, then $A\cap B=\{1\}$ and
\[
|AB|=\frac{|A||B|}{|A\cap B|}=p^2,
\]
contradicting $|G|=p(p-1)<p^2$.
Thus $A$ is unique and therefore normal.

Choose $t\in G$ whose image generates the cyclic quotient $G/A\cong C_{p-1}$.
Since $G$ is non-abelian and $A$ has prime order, one has
\[
F:=A\cap \Z(G)=1;
\]
indeed, if $F\neq 1$ then $A\le \Z(G)$, and the cyclic quotient would force $G$ to be abelian.

Now fix $1\le j\le p-2$.
If $t^j$ centralized $A$, then $t^j$ would commute with both $A$ and $t$, hence would lie in $\Z(G)$.
But $t^j\notin A$ because $tA$ has order $p-1$ in $G/A$, contradicting the assumption that every element of $G\setminus A$ is noncentral.
Therefore the automorphism induced by $t^j$ on $A\cong C_p$ is nontrivial.
Any nontrivial automorphism of $C_p$ fixes only the identity, so
\[
\C_A(t^j)=1=F \qquad (1\le j\le p-2).
\]
Thus Theorem~\ref{prop:Pk-normal-index-p} applies with $\omega=p-1$, $n=p$, and $f=1$, yielding the displayed formula.

For $P_2(G)$ we obtain
\[
P_2(G)=\frac{1}{(p-1)^2}+\left(1-\frac{1}{(p-1)^2}\right)\frac{1}{p}
=\frac{1}{p-1}.
\]
\end{proof}

\begin{example}
    Let $G$ be the following non-abelian group of order 20.
    $$G=\langle a,b \mid a^5=b^4=1,\ bab^{-1}=a^2\rangle.$$
    Then the subgroup $A=\langle a\rangle$ is an abelian normal subgroup of $G$, such that $G/A\cong C_4$ and $Z(G)=\{1\}$. Thus, every $x\in G\setminus A$ is not a central element. Hence, applying Corollary \ref{pp-1 groups},
    $$P_r(G)=\frac{1}{4^r}+\left(1-\frac{1}{4^r}\right)\left(\frac{1}{5}\right)^{r-1},$$
    and
    $$P_2(G)=\frac{1}{4}.$$
\end{example}

\begin{corollary}[Normal abelian subgroup of prime index]\label{prop:Pk-normal-index-p1}
Let $G$ be a finite non-abelian group, and let $A\trianglelefteq G$ be an abelian normal subgroup such that $[G:A]=p$ is prime.
Choose $t\in G$ whose image generates $G/A$.
Then
\[
\C_A(t^i)=\C_A(t)=A\cap \Z(G)\qquad (1\le i\le p-1).
\]
Consequently, Theorem~\ref{prop:Pk-normal-index-p} applies to $G$.
\end{corollary}

\begin{proof}
Since $[G:A]=p$, the quotient $G/A$ is cyclic of order $p$.
Let $\varphi\in \Aut(A)$ be conjugation by $t$.
If $\varphi=\mathrm{id}$, then $t$ centralizes $A$ and $G=\langle A,t\rangle$ is abelian, contrary to hypothesis.
Hence $\varphi$ has order $p$.

For each $1\le i\le p-1$, the automorphism $\varphi^i$ generates the same cyclic subgroup of $\Aut(A)$ as $\varphi$, because $\gcd(i,p)=1$.
Therefore
\[
\Fix_A(\varphi^i)=\Fix_A(\varphi).
\]
But
\[
\Fix_A(\varphi)=\C_A(t)=A\cap \Z(G),
\]
so
\[
\C_A(t^i)=A\cap \Z(G)\qquad (1\le i\le p-1),
\]
as required.
\end{proof}

\begin{example}[Normality in Corollary~\ref{prop:Pk-normal-index-p1} cannot be dropped]\label{ex:S3-index3-formula-fails}
Let $G=S_3$ and let $A=\langle(12)\rangle\cong C_2$.
Then $A$ is abelian of prime index $3$, but it is not normal.
If one naively tried to extend Corollary~\ref{prop:Pk-normal-index-p1} to this situation using $p=3$, $n=2$, and $|A\cap Z(G)|=1$, one would obtain
\[
P_2(G)\stackrel{?}{=} \frac{1}{3^2}+\left(1-\frac{1}{3^2}\right)\frac12=\frac59.
\]
In fact
\[
P_2(S_3)=\frac{k(S_3)}{|S_3|}=\frac{3}{6}=\frac12.
\]
So the normality hypothesis in Corollary~\ref{prop:Pk-normal-index-p1} is essential.
\end{example}

\begin{corollary}[Normality of index-$p$ subgroups when $p$ is minimal]\label{cor:index-p-normality}
Let $G$ be a finite group with $|G|=pn$, where $p$ is the smallest prime divisor of $|G|$.
If $G$ contains a subgroup $A$ of index $p$, then $A\trianglelefteq G$.
If moreover $A$ is abelian, then the formula of Theorem~\ref{prop:Pk-normal-index-p} applies.
\end{corollary}

\begin{proof}
    The action of $G$ on the left cosets $G/A$ yields a homomorphism $\rho\colon G\to S_p$.
Its image is transitive, hence $p\mid |\rho(G)|$.
Every prime divisor of $|\rho(G)|$ divides $|G|$, hence is at least $p$, while also dividing $p!$, hence is at most $p$.
Thus $\rho(G)$ is a $p$-group.
Since $p$ occurs only once in $p!$, we have $|\rho(G)|=p$ and $\ker(\rho)$ has index $p$.
Moreover $\ker(\rho)\subseteq A$, so $\ker(\rho)=A$ and $A\trianglelefteq G$.
\end{proof}

\begin{example}[The smallest-prime hypothesis is essential]\label{ex:index-p-not-normal}
Again let $G=S_3$ and $A=\langle(12)\rangle$.
Then $[G:A]=3$, but $3$ is not the smallest prime divisor of $|G|=6$, and indeed $A$ is not normal.
Thus Corollary~\ref{cor:index-p-normality} fails without the assumption that $p$ is the smallest prime divisor of $|G|$.
Equivalently, Theorem~\ref{thm:gap-index-p-abelian} cannot be extended to arbitrary abelian subgroups of prime index:
here
\[
|A\cap Z(G)|=1>\frac{|A|}{3}=\frac23.
\]
\end{example}

\begin{corollary}[All higher commuting probabilities for dihedral groups]\label{cor:dihedral-all-Pr}
Let $G=D_{2n}$. For every $r\ge 2$,
\[
P_r(D_{2n})
=\frac{1}{2^r}+\left(1-\frac{1}{2^r}\right)\left(\frac{f}{n}\right)^{r-1},
\]
where $f=1$ if $n$ is odd and $f=2$ if $n$ is even.
In particular, $\lim_{n\to\infty}P_r(D_{2n})=2^{-r}$ for each fixed $r$.
\end{corollary}

\begin{proof}
Apply Theorem~\ref{prop:Pk-normal-index-p} with $\omega=2$ and $A=\langle a\rangle\cong C_n$.
The fixed subgroup $F=\C_A(b)$ is trivial if $n$ is odd and has order $2$ if $n$ is even.
\end{proof}

\begin{corollary}[Groups of order $pq$]\label{cor:order-pq}
Let $p<q$ be primes and let $G$ be a group of order $pq$.
Then either $G$ is cyclic (hence abelian), in which case $P_r(G)=1$ for all $r\ge 2$,
or else $G$ is non-abelian and for every $r\ge 2$,
\[
P_r(G)=\frac{1}{p^r}+\left(1-\frac{1}{p^r}\right)\frac{1}{q^{\,r-1}}.
\]
In particular,
\[
P_2(G)=\frac{1}{p^2}+\left(1-\frac{1}{p^2}\right)\frac{1}{q},
\qquad
P_3(G)=\frac{1}{p^3}+\left(1-\frac{1}{p^3}\right)\frac{1}{q^{2}}.
\]
\end{corollary}

\begin{proof}
If $G$ is abelian, then it is cyclic of order $pq$ and $P_r(G)=1$ for all $r\ge 2$.
Assume $G$ is non-abelian.
By Cauchy's theorem, $G$ contains a subgroup $A$ of order $q$, hence of index $p$.
Since $p$ is the smallest prime divisor of $|G|$, Corollary~\ref{cor:index-p-normality} shows that $A\trianglelefteq G$.
Thus $A\cong C_q$ is abelian with $[G:A]=p$.
Because $G$ is non-abelian, $A$ is not central, so $A\cap \Z(G)=1$.
Now Theorem~\ref{prop:Pk-normal-index-p} with $n=|A|=q$ and $f=|A\cap \Z(G)|=1$ gives the stated formula.
\end{proof}

\begin{corollary}\label{pp+1 groups}
    Let $G$ be a non-abelian group of order $p(p+1)$ which contains a normal abelian subgroup $A$ of order $p+1$.
    Then, by Corollary~\ref{prop:Pk-normal-index-p1},
    \[
P_r(G)=\frac{1}{p^r}+\left(1-\frac{1}{p^r}\right)\frac{1}{(p+1)^{\,r-1}}.
\]
In particular,
\[
P_2(G)=\frac{1}{p}.
   \]
\end{corollary}

\begin{proof}
 Let $t\in G\setminus A$.
 Since $[G:A]=p$, we have $G=\bigcup_{k=0}^{p-1} t^kA$ and hence $G=\langle A,t\rangle$.
 Because $G$ is non-abelian, $A$ is not central, so choose $a\in A\setminus Z(G)$.
 The conjugation action of $t$ on $A$ has order dividing $p$, and the orbit of $a$ under this action has size $p$.
 Thus the $p$ elements
 \[
 \{t^{-k}at^k : 0\le k\le p-1\}
\]
 are distinct and lie in $A$.
 Since $|A|=p+1$, the only element of $A$ fixed by conjugation is $1$, so $A\cap Z(G)=1$.
 Therefore Corollary~\ref{prop:Pk-normal-index-p1} gives the formula for $P_r(G)$.
 The identity $P_2(G)=1/p$ follows by the same computation as in Corollary~\ref{pp-1 groups}.
\end{proof}

\begin{corollary}[Hierarchy collapse for normal abelian index-$\omega$ groups]\label{cor:hierarchy-collapse}
Assume the hypotheses of Theorem~\ref{prop:Pk-normal-index-p} and set $\alpha:=f/n$.
Then for every $r\ge 2$,
\[
P_r(G)=\frac{1}{\omega^r}+\left(1-\frac{1}{\omega^r}\right)\alpha^{r-1},
\qquad\text{and}\qquad
\alpha=\frac{P_2(G)-\omega^{-2}}{1-\omega^{-2}}.
\]
More generally, for each fixed $r\ge 2$ one can recover $\alpha^{r-1}$ from $P_r(G)$ via
\[
\alpha^{r-1}=\frac{P_r(G)-\omega^{-r}}{1-\omega^{-r}}.
\]
Hence, within this family of groups, the commuting hierarchy $\{P_r(G)\}_{r\ge 2}$ is completely determined by any single value $P_r(G)$ (and in particular by $P_2(G)$).
\end{corollary}

\begin{corollary}\label{cor:P3-index-p-limit}
Fix a prime $p$, and let $(G_n)_{n\ge 1}$ be a sequence of finite groups with $|G_n|=pn$ such that $p$ is the smallest prime divisor of $|G_n|$
and $G_n$ contains an abelian subgroup $A_n$ of order $n$.
Set $f_n:=|A_n\cap \Z(G_n)|$.
Then for every $r\ge 2$,
\[
P_r(G_n)=\frac{1}{p^r}+\left(1-\frac{1}{p^r}\right)\left(\frac{f_n}{n}\right)^{r-1}.
\]
In particular, for each fixed $r\ge 2$ one has
\[
\lim_{n\to\infty} P_r(G_n)=\frac{1}{p^r}
\quad\Longleftrightarrow\quad
\frac{f_n}{n}\longrightarrow 0.
\]
More generally, if $f_n/n\to \alpha\in[0,1]$, then for each fixed $r\ge 2$,
\[
\lim_{n\to\infty} P_r(G_n)
=\frac{1}{p^r}+\left(1-\frac{1}{p^r}\right)\alpha^{r-1}.
\]
\end{corollary}

\begin{theorem}[Gap and extremal rigidity for non-abelian index-$p$ abelian extensions]\label{thm:gap-index-p-abelian}
Let $G$ be a finite group with $|G|=pn$, where $p$ is the smallest prime divisor of $|G|$.
Assume that $G$ contains an abelian subgroup $A$ of order $n$ (equivalently, $[G:A]=p$), and suppose that $G$ is non-abelian.
Set
\[
F:=A\cap \Z(G),\qquad f:=|F|.
\]
Then $f\le n/p$, i.e.
\[
\frac{|A\cap \Z(G)|}{|A|}\le \frac{1}{p}.
\]
Consequently, for every $r\ge 2$,
\begin{equation}\label{eq:gap-Pk-index-p}
P_r(G)\le \frac{p^r+p^{r-1}-1}{p^{2r-1}}.
\end{equation}
Moreover, the following are equivalent:
\begin{enumerate}
\item\label{it:extremal-f} $f=n/p$ (equivalently, $n/f=p$);
\item\label{it:extremal-Pk} equality holds in \eqref{eq:gap-Pk-index-p} for some (equivalently, for every) $r\ge 2$;
\item\label{it:extremal-quot} $G/\Z(G)\cong C_p\times C_p$;
\item\label{it:extremal-maxab} $G$ has exactly $p+1$ maximal abelian subgroups containing $\Z(G)$.
Equivalently, these are $A$ and the $p$ distinct centralizers $\C_G(x)$ with $x\in G\setminus A$.
\end{enumerate}
\end{theorem}

\begin{proof}
By Corollary~\ref{cor:index-p-normality}, $A\trianglelefteq G$.
Choose $t\in G\setminus A$, so $G=\langle A,t\rangle$ and conjugation by $t$ defines an automorphism
$\varphi\in \Aut(A)$ of order $p$.
Since $G$ is non-abelian, $\varphi\neq \mathrm{id}$.

Consider the homomorphism
\[
\delta:A\to A,\qquad \delta(a)=a^{-1}\varphi(a)=a^{-1}t a t^{-1}.
\]
Its kernel is $\Fix_A(\varphi)=\C_A(t)=F$, hence $|F|=|A|/|\delta(A)|$.
Since $\varphi\neq \mathrm{id}$, the image $\delta(A)$ is a nontrivial subgroup of $A$.
Every prime divisor of $|A|=n$ is at least $p$ (since $p$ is the smallest prime divisor of $|G|$ and $A\le G$).
Therefore any nontrivial subgroup of $A$ has order divisible by some prime $\ge p$, and hence has order at least $p$.
Thus $|\delta(A)|\ge p$, and therefore $f=|F|\le n/p$.

Now Corollary~\ref{prop:Pk-normal-index-p1} shows that Theorem~\ref{prop:Pk-normal-index-p} applies.
Therefore, for each fixed $r\ge 2$ the function $x\mapsto x^{r-1}$ is increasing on $[0,1]$, so from $f/n\le 1/p$ we get
\[
P_r(G)\le \frac{1}{p^r}+\left(1-\frac{1}{p^r}\right)\left(\frac{1}{p}\right)^{r-1}
=\frac{p^r+p^{r-1}-1}{p^{2r-1}},
\]
which is \eqref{eq:gap-Pk-index-p}.

We next record that in the non-abelian case one has $\Z(G)\subseteq A$, hence $\Z(G)=A\cap \Z(G)=F$.
Indeed, if $z\in \Z(G)\setminus A$ then $zA$ is a nontrivial element of $G/A\cong C_p$, hence a generator.
Thus $G=\langle A,z\rangle$.
Since $z$ is central and $A$ is abelian, this forces $G$ to be abelian, contradicting the hypothesis.

\smallskip
\noindent\emph{(\ref{it:extremal-f}) $\Longleftrightarrow$ (\ref{it:extremal-Pk}).}
By Corollary~\ref{prop:Pk-normal-index-p1}, the exact formula \eqref{eq:Pk-index-p} applies, and for each fixed $r\ge 2$ the value of $P_r(G)$ is a strictly increasing function of $f/n\in(0,1]$.
Since $f/n\le 1/p$, equality holds in \eqref{eq:gap-Pk-index-p} for some (equivalently, for every) $r\ge 2$
if and only if $f/n=1/p$, i.e.\ $f=n/p$.

\smallskip
\noindent\emph{(\ref{it:extremal-f}) $\Longleftrightarrow$ (\ref{it:extremal-quot}).}
If $f=n/p$, then $\Z(G)=F$ and so $|G:\Z(G)|=|G:F|=p\cdot(n/f)=p^2$.
Since $G$ is non-abelian this forces $G/\Z(G)\cong C_p\times C_p$.
Conversely, if $G/\Z(G)\cong C_p\times C_p$ then $|G:\Z(G)|=p^2$ and $\Z(G)=F$ as above, so $f=n/p$.

\smallskip
\noindent\emph{(\ref{it:extremal-f}) $\Longleftrightarrow$ (\ref{it:extremal-maxab}).}
For $x\in G\setminus A$ write $M_x:=\C_G(x)$.
Then $|M_x|=pf$ (see the proof of Theorem~\ref{prop:Pk-normal-index-p}), and in particular $M_x$ is abelian.
If $y\in G\setminus A$ commutes with $x$, then $y\in M_x$ and $M_x\subseteq \C_G(y)=M_y$.
Comparing orders gives $M_x=M_y$.
Thus the sets $M_x\setminus A$ partition $G\setminus A$ into blocks of size $|M_x\setminus A|=(p-1)f$,
so there are exactly
\[
d:=\frac{|G\setminus A|}{(p-1)f}=\frac{(p-1)n}{(p-1)f}=\frac{n}{f}
\]
distinct subgroups among the $M_x$; denote them $M_1,\dots,M_d$.
Every maximal abelian subgroup containing $\Z(G)=F$ is either $A$ or one of the $M_i$:
indeed, if $B$ is abelian, contains $F$, and is not contained in $A$, then choosing $x\in B\setminus A$ gives
$B\le \C_G(x)=M_x$, hence $B=M_x$ by maximality.
Therefore $G$ has exactly $d+1$ maximal abelian subgroups containing $\Z(G)$.
In particular, $G$ has exactly $p+1$ such subgroups if and only if $d=p$, i.e.\ $f=n/p$.
\end{proof}

\medskip
\noindent\textit{Remark.}
In the extremal case $G/\Z(G)\cong C_p\times C_p$, the $p+1$ maximal abelian subgroups containing $\Z(G)$ correspond bijectively to the $p+1$ one-dimensional subspaces of the $2$-dimensional $\mathbb{F}_p$-vector space $G/\Z(G)$.
\medskip

\begin{proposition}[Realizing fixed-point proportions]\phantomsection\label{prop:alpha-realize}
\begin{enumerate}
\item Fix a prime $p$ and an integer $d\ge 1$.
There exists a non-abelian group $G$ of order $p^{2d+1}$ with an abelian subgroup $A$ of index $p$ such that
\[
\frac{|A\cap \Z(G)|}{|A|}=\frac{1}{p^d}.
\]
\item Let $p$ and $q$ be primes with $p\mid (q-1)$, and let $e\ge 1$.
There exists a non-abelian group $G$ of order $pq^e$ with an abelian subgroup $A$ of index $p$ such that
\[
\frac{|A\cap \Z(G)|}{|A|}=\frac{1}{q^e}.
\]
\end{enumerate}
\end{proposition}

\begin{proof}
(1) Let $A\cong (C_p)^{2d}$ and identify it with the $\mathbb F_p$-vector space $V=\mathbb F_p^{2d}$.
Let $\varphi\in \mathrm{GL}(V)$ be the block diagonal matrix with $d$ Jordan blocks
$\left(\begin{smallmatrix}1&1\\0&1\end{smallmatrix}\right)$.
Then $\varphi$ has order $p$, and its fixed space has dimension $d$, hence $|\Fix_A(\varphi)|=p^d=|A|/p^d$.
Let $G:=A\rtimes_\varphi C_p$.
Then $G$ is non-abelian and satisfies the hypotheses of Corollary~\ref{prop:Pk-normal-index-p1} with $A\cap \Z(G)=\Fix_A(\varphi)$.

(2) Let $A:=C_{q^e}=\langle a\rangle$.
Since $(\mathbb Z/q^e\mathbb Z)^\times$ is cyclic of order $q^{e-1}(q-1)$ and $p\mid (q-1)$, there exists $u\in (\mathbb Z/q^e\mathbb Z)^\times$ of order $p$.
Let $G:=A\rtimes C_p$ where a generator $t$ of $C_p$ acts by $t a t^{-1}=a^u$.
Then $\varphi(a)=a^u$ has order $p$ and is nontrivial modulo $q$, so $\gcd(u-1,q^e)=1$.
Hence $\Fix_A(\varphi)=1$, i.e.\ $|A\cap \Z(G)|=1$, and $G$ is non-abelian.
\end{proof}

\medskip
\noindent\textit{Remark (Counterexamples to a naive limit conjecture).}
Motivated by Corollary~\ref{dihedral-P3-lim}, an earlier draft proposed the following (false) ``naive'' statement:
whenever $|G|=pn$ with $p$ the smallest prime divisor of $|G|$ and $G$ contains an abelian subgroup $A$ of order $n$,
one has $\lim_{n\to\infty}P_3(G)=1/p^3$.
Theorem~\ref{prop:Pk-normal-index-p} shows that, more generally, for each fixed $r\ge 2$ the value of $P_r(G)$ is governed by the
\emph{fixed-point proportion}
\[
\frac{|A\cap \Z(G)|}{|A|},
\]
and without an assumption forcing this ratio to tend to $0$ the limit can be strictly larger than $1/p^r$.
Here are three concrete families illustrating this phenomenon.

\begin{itemize}
\item \emph{Abelian groups.} If $G_n$ is abelian (for instance $G_n\cong C_{pn}$), then $P_r(G_n)=1$ for all $r\ge 2$ and all $n$.

\item \emph{A non-abelian family with constant limit.} Let $p=2$ and set $G_m:=S_3\times C_m$ with $m$ odd.
Then $|G_m|=6m=2(3m)$, and $A_m:=C_3\times C_m\le G_m$ is abelian of order $3m$.
Since $\Z(G_m)=\Z(S_3)\times C_m\cong C_m$, we have $A_m\cap \Z(G_m)\cong C_m$, hence $f_m/n=1/3$.
Theorem~\ref{prop:Pk-normal-index-p} yields, for every $r\ge 2$,
\[
P_r(G_m)
=\frac{1}{2^r}+\left(1-\frac{1}{2^r}\right)\left(\frac13\right)^{r-1}.
\]
In particular, $P_3(G_m)=2/9$ for all $m$, so $\lim_{m\to\infty}P_3(G_m)=2/9\neq 1/8$.

\item \emph{A non-abelian $p$-group family.} Fix a prime $p$ and an integer $e\ge 2$.
Let $G_e:=C_{p^e}\rtimes C_p$ where a generator $t$ of $C_p$ acts on $C_{p^e}=\langle a\rangle$ by
$t a t^{-1}=a^{\,1+p^{e-1}}$.
Then $|G_e|=p^{e+1}=p\cdot p^e$, and $A:=\langle a\rangle$ is an abelian subgroup of order $p^e$.
A direct check shows that $A\cap \Z(G_e)=\langle a^p\rangle$ has order $p^{e-1}$, so $f/n=1/p$.
Therefore Theorem~\ref{prop:Pk-normal-index-p} gives the constant value, for every $r\ge 2$,
\[
P_r(G_e)=\frac{1}{p^r}+\left(1-\frac{1}{p^r}\right)\frac{1}{p^{r-1}},
\]
independent of $e$.
\end{itemize}

\noindent More generally, along any sequence $(G_n,A_n)$ as in Corollary~\ref{cor:P3-index-p-limit}, whenever the limit
\[
\alpha:=\lim_{n\to\infty}\frac{|A_n\cap \Z(G_n)|}{|A_n|}
\]
exists in $[0,1]$, one necessarily has
\[
\lim_{n\to\infty}P_r(G_n)
=\frac{1}{p^r}+\left(1-\frac{1}{p^r}\right)\alpha^{r-1}
\qquad (r\ge 2\ \text{fixed}).
\]
It would be interesting to understand which values of $\alpha$ can actually occur; Proposition~\ref{prop:alpha-realize}
shows that at least $\alpha=p^{-d}$ (for any $d\ge 1$) and $\alpha=q^{-e}$ (for primes $q\equiv 1\ (\mathrm{mod}\,p)$) are realizable.

\section{Sharp bounds for \texorpdfstring{$P_r$}{Pr} and extremal groups}\label{sec:bounds}

\subsection{A sharp bound for \texorpdfstring{$P_2(G)$}{P2(G)} (general \texorpdfstring{$p$}{p})}
Let $p$ be the smallest prime dividing $|G|$.

\begin{lemma}\label{lem:centerindex}
If $G$ is finite and non-abelian, then $|G:\Z(G)|\ge p^2$, hence $|\Z(G)|/|G|\le 1/p^2$.
\end{lemma}

\begin{proof}
The quotient $G/\Z(G)$ is cyclic if and only if $G$ is abelian.
If $G$ is non-abelian, then $G/\Z(G)$ is a nontrivial \emph{noncyclic} finite group.
The smallest order of a noncyclic finite group whose order divides $|G|$ is at least $p^2$,
since any group of order $p$ is cyclic.
Thus $|G:\Z(G)|\ge p^2$.
\end{proof}

\begin{lemma}\label{lem:centralizerindexp}
For every noncentral $x\in G$, the index $[G:\C_G(x)]$ is a nontrivial divisor of $|G|$,
hence $[G:\C_G(x)]\ge p$, so $|\C_G(x)|\le |G|/p$.
\end{lemma}

\begin{proof}
$\C_G(x)$ is a proper subgroup iff $x\notin \Z(G)$, so its index is an integer $>1$
dividing $|G|$. The least such integer is at least the smallest prime divisor $p$ of $|G|$.
\end{proof}

\begin{theorem}[Sharp $P_2$ bound]\label{thm:P2bound}
Let $G$ be a finite non-abelian group, and let $p$ be the smallest prime dividing $|G|$.
Then
\[
P_2(G)\le \frac{p^2+p-1}{p^3}.
\]
Equality holds if and only if $G/\Z(G)\cong C_p\times C_p$.
\end{theorem}

\begin{proof}
Write $\alpha:=|\Z(G)|/|G|$. Then
\[
|G|^2P_2(G)=\sum_{x\in G} |\C_G(x)|
=\sum_{x\in \Z(G)}|G|+\sum_{x\notin \Z(G)}|\C_G(x)|.
\]
The first sum equals $|\Z(G)||G|=\alpha |G|^2$.
By Lemma~\ref{lem:centralizerindexp}, for $x\notin\Z(G)$ one has $|\C_G(x)|\le |G|/p$.
Thus
\[
|G|^2P_2(G)\le \alpha |G|^2 + (1-\alpha)|G|^2\cdot \frac{1}{p},
\quad\text{i.e.}\quad
P_2(G)\le \alpha + \frac{1-\alpha}{p}.
\]
By Lemma~\ref{lem:centerindex} we have $\alpha\le 1/p^2$. Since the right-hand side is
increasing in $\alpha$, it is maximized at $\alpha=1/p^2$, giving
\[
P_2(G)\le \frac{1}{p^2}+\Big(1-\frac{1}{p^2}\Big)\frac{1}{p}
=\frac{p^2+p-1}{p^3}.
\]

For equality, we must have $\alpha=1/p^2$, so $|G:\Z(G)|=p^2$.
Also the estimate $|\C_G(x)|\le |G|/p$ must be an equality for all $x\notin\Z(G)$, so
$[G:\C_G(x)]=p$ for all noncentral $x$.
If $|G:\Z(G)|=p^2$, then $G/\Z(G)$ has order $p^2$. If it were cyclic then $G$ would be abelian;
hence $G/\Z(G)\cong C_p\times C_p$.

Conversely, if $G/\Z(G)\cong C_p\times C_p$, then $\alpha=1/p^2$ and every noncentral element
has centralizer of index $p$ (because its image spans a subgroup of order $p$ in the quotient).
Substituting in the above computation gives equality.
\end{proof}

The groups $G$ that have $P_2(G)\geq \frac{1}{2}$ have already been fully classified in \cite{Lescot1995}, where the possible values of $P_2(G)$ are in one of the following cases (i.e., if $G$ is not one of the following types of groups, then $P_2(G)<\frac{1}{2}$).
\begin{itemize}
    \item $P_2(G)=1$ ~if and only if $G$ is abelian,
    \item For a positive integer $k$, ~$P_2(G)=\frac{1} {2}+\frac{1}{2^{2k+1}}$ ~if and only if $G$ is isoclinic to one of the two non-abelian supersolvable 2-groups of order $2^{2k+1}$, with $G/Z(G)\simeq C_2^{2k}$;
    \item $P_2(G)=\frac{1}{2}$ ~if and only if $G$ is isoclinic to $S_3$.
\end{itemize}

\subsection{A general bound for \texorpdfstring{$P_r(G)$}{Pr(G)}}
\begin{lemma}[A one-step inequality]\label{lem:onestep}
Let $G$ be a finite group, $r\ge 2$, and $p$ the smallest prime dividing $|G|$.
With $\alpha=|\Z(G)|/|G|$,
\[
P_r(G)\le \alpha\,P_{r-1}(G) + \frac{1-\alpha}{p^{\,r-1}}.
\]
\end{lemma}

\begin{proof}
Use Lemma~\ref{lem:recursion}:
\[
P_r(G)=\frac{1}{|G|^r}\sum_{x\in G} |\C_G(x)|^{r-1}P_{r-1}(\C_G(x)).
\]
If $x\in \Z(G)$ then $\C_G(x)=G$, contributing $|G|^{r-1}P_{r-1}(G)$.
If $x\notin \Z(G)$ then $|\C_G(x)|\le |G|/p$ by Lemma~\ref{lem:centralizerindexp}, while
$P_{r-1}(\C_G(x))\le 1$. Hence the sum is at most
\[
|\Z(G)|\cdot |G|^{r-1}P_{r-1}(G) + (|G|-|\Z(G)|)\cdot (|G|/p)^{r-1}.
\]
Divide by $|G|^r$.
\end{proof}

\begin{proposition}[A two-block inequality]\label{prop:twoblock}
Let $G$ be a finite group, let $p$ be the smallest prime dividing $|G|$, and write
$\alpha:=|\Z(G)|/|G|$.
Then for all integers $n,m\ge 1$ one has the two-sided estimate
\[
\alpha^{n}\,P_m(G)\ \le\ P_{n+m}(G)\ \le\ \alpha^{n}\,P_m(G) + \frac{P_n(G)-\alpha^{n}}{p^{m}}.
\]
Equivalently,
\[
0\le P_{n+m}(G)-\alpha^{n}P_m(G)\le \frac{P_n(G)-\alpha^{n}}{p^{m}}.
\]
The analogous inequalities with the roles of $n$ and $m$ interchanged also hold.
\end{proposition}

\begin{proof}
Let $(x_1,\dots,x_{n+m})$ be uniformly random in $G^{n+m}$ and write
$X:=(x_1,\dots,x_n)$ and $Y:=(x_{n+1},\dots,x_{n+m})$.
Let $E_k$ denote the event that a $k$-tuple commutes pairwise.
Finally, let $A$ be the event that $X\in \Z(G)^n$ (i.e. all coordinates of $X$ are central).

For the lower bound, note that $A\cap E_m(Y)\subseteq E_{n+m}(X,Y)$: if $X$ is central and $Y$ is
pairwise commuting, then the combined $(n{+}m)$-tuple is pairwise commuting.
Since $A$ depends only on $X$ and $E_m(Y)$ depends only on $Y$, they are independent, hence
\[
P_{n+m}(G)=\Prob(E_{n+m}(X,Y))\ge \Prob(A)\Prob(E_m(Y))=\alpha^{n}P_m(G).
\]

For the upper bound, decompose
\[
\Prob(E_{n+m})=\Prob(E_{n+m}\cap A)+\Prob(E_{n+m}\cap A^c).
\]
As above, $\Prob(E_{n+m}\cap A)=\Prob(A)\Prob(E_m(Y))=\alpha^{n}P_m(G)$.

On $A^c\cap E_n(X)$, there exists at least one noncentral coordinate of $X$; fix the least index $i$
such that $x_i\notin \Z(G)$.
If $E_{n+m}$ holds, then every coordinate of $Y$ must commute with $x_i$, hence must lie in $\C_G(x_i)$.
By Lemma~\ref{lem:centralizerindexp} we have $|\C_G(x_i)|/|G|\le 1/p$.
Therefore, for each fixed $X$ in $E_n(X)\cap A^c$,
\[
\Prob(E_{n+m}\mid X)\le \Prob\big(Y\in \C_G(x_i)^m\big)\le \Big(\frac{|\C_G(x_i)|}{|G|}\Big)^m\le \frac{1}{p^{m}}.
\]
Integrating over $X$ gives
\[
\Prob(E_{n+m}\cap A^c)\le \frac{1}{p^{m}}\Prob(E_n(X)\cap A^c)=\frac{P_n(G)-\alpha^{n}}{p^{m}}.
\]
Combining with the $A$-contribution yields the stated upper bound.

Finally, the symmetric statement follows by the same argument with $X$ and $Y$ swapped.
\end{proof}

\begin{remark}
Proposition~\ref{prop:twoblock} refines the one-step inequality in Lemma~\ref{lem:onestep}:
taking $n=1$ gives
$P_{m+1}(G)\le \alpha P_m(G)+(1-\alpha)/p^{m}$.
Moreover, in the extremal situation $G/\Z(G)\cong C_p\times C_p$ (Theorem~\ref{thm:Prbound}), every
noncentral element has abelian centralizer of index $p$, and the upper bound in
Proposition~\ref{prop:twoblock} is attained for many $(n,m)$.
\end{remark}

\begin{corollary}[Exponential approach to the central contribution]\label{cor:expcentral}
With notation as in Proposition~\ref{prop:twoblock}, for all integers $n,m\ge 1$ one has
\[
\bigl|P_{n+m}(G)-\alpha^{n}P_m(G)\bigr|\le \frac{P_n(G)-\alpha^{n}}{p^{m}}.
\]
In particular, for each fixed $n$ the ``noncentral extension'' contribution
\(P_{n+m}(G)-\alpha^{n}P_m(G)\) decays at least exponentially fast in $m$ with base $p$.
Equivalently,
\[
\Prob\bigl(E_{n+m}(X,Y)\cap A^c\bigr)\le \frac{P_n(G)-\alpha^{n}}{p^{m}},
\]
in the notation of Proposition~\ref{prop:twoblock}.
If $\Z(G)=1$ (so $\alpha=0$) this simplifies to
\(
P_{n+m}(G)\le P_n(G)/p^{m}
\),
showing that, for fixed $n$, the probability of a commuting $(n{+}m)$-tuple decays at least like $p^{-m}$.
\end{corollary}

\begin{remark}[Stability heuristic]\label{rem:stability-heuristic}
Proposition~\ref{prop:twoblock} and Corollary~\ref{cor:expcentral} give a clean ``two-phase'' mechanism
for producing long commuting tuples.
To extend a commuting $n$-tuple by $m$ additional commuting elements, either
(i) the first block lies in $\Z(G)^n$, in which case the cost of the extension is exactly $P_m(G)$, or
(ii) some noncentral element appears in the first block, in which case all $m$ new elements are forced into
a proper centralizer of index at least $p$, costing a factor at most $p^{-m}$.
Thus any noticeable excess
\(
P_{n+m}(G)-\alpha^{n}P_m(G)\ge \varepsilon
\)
forces
\(
p^{m}\le (P_n(G)-\alpha^{n})/\varepsilon
\),
so for fixed $n$ such an excess can persist only for small $m$ or very small $p$.
This is the probabilistic mechanism behind the sharp bound in Theorem~\ref{thm:Prbound} and its rigidity of equality:
the extremal isoclinism class $G/\Z(G)\cong C_p\times C_p$ is exactly the situation in which the ``noncentral''
branch (ii) is as large as possible, since every noncentral centralizer has index $p$ and is abelian.
\end{remark}

\begin{corollary}[Induction step]\label{cor:induction-step}
Taking $n=1$ and $m=r-1$ in Proposition~\ref{prop:twoblock} yields
\[
P_r(G)\le \alpha\,P_{r-1}(G)+\frac{1-\alpha}{p^{r-1}},
\]
which is exactly the one-step estimate used in the inductive proof of Theorem~\ref{thm:Prbound}.
\end{corollary}

\begin{theorem}[Known sharp general $P_r$ bound, recalled]\label{thm:Prbound}
Let $G$ be a finite non-abelian group, and let $p$ be the smallest prime dividing $|G|$.
Then for every $r\ge 2$,
\[
P_r(G)\le \frac{p^r+p^{r-1}-1}{p^{2r-1}}.
\]
Equality holds if and only if $G/\Z(G)\cong C_p\times C_p$.
This is the known sharp bound for the multiple commutativity degree (translated into our notation for $P_r$);
see for instance \cite[Prop.~2.4]{RezaeiNiroomandErfanian2014}. We recall a short proof only to fix notation and equality cases for the later rigidity applications.
\end{theorem}

\begin{proof}
We argue by induction on $r\ge 2$.

For $r=2$ this is Theorem~\ref{thm:P2bound}.

Assume the claim for $r-1$ and let $G$ be non-abelian. Set
\[
M_{r-1}:=\frac{p^{r-1}+p^{r-2}-1}{p^{2r-3}}.
\]
By the induction hypothesis, $P_{r-1}(G)\le M_{r-1}$ for every non-abelian $G$.
Lemma~\ref{lem:onestep} gives
\[
P_r(G)\le \alpha\,P_{r-1}(G)+\frac{1-\alpha}{p^{r-1}}
\le \alpha\,M_{r-1}+\frac{1-\alpha}{p^{r-1}}.
\]
Since $M_{r-1}>1/p^{r-1}$, the right-hand side is increasing in $\alpha$.
By Lemma~\ref{lem:centerindex}, $\alpha\le 1/p^2$ for non-abelian $G$, hence
\[
P_r(G)\le \frac{1}{p^2}M_{r-1}+\Big(1-\frac{1}{p^2}\Big)\frac{1}{p^{r-1}}
=\frac{p^r+p^{r-1}-1}{p^{2r-1}},
\]
a direct simplification.

For equality we need equality in each step:
(i) $\alpha=1/p^2$, so $|G:\Z(G)|=p^2$ and $G/\Z(G)\cong C_p\times C_p$ as before;
(ii) $P_{r-1}(G)=M_{r-1}$, which by induction forces the same structure;
(iii) in Lemma~\ref{lem:onestep}, for each $x\notin\Z(G)$ we need $|\C_G(x)|=|G|/p$ and
$P_{r-1}(\C_G(x))=1$, i.e.\ $\C_G(x)$ is abelian. When $G/\Z(G)\cong C_p\times C_p$, every
proper subgroup containing $\Z(G)$ has cyclic image in the quotient, hence is abelian,
so these conditions hold.

Conversely, if $G/\Z(G)\cong C_p\times C_p$, choose a subgroup $A$ with $\Z(G)\subseteq A$ and $[G:A]=p$.
Then $A$ is abelian and normal, and
\[
\frac{|A\cap \Z(G)|}{|A|}=\frac{|\Z(G)|}{|G|/p}=\frac1p.
\]
Hence Corollary~\ref{prop:Pk-normal-index-p1} applies, and the earlier exact formula \eqref{eq:Pk-index-p} from Theorem~\ref{prop:Pk-normal-index-p} gives
\[
P_r(G)=\frac{1}{p^r}+\left(1-\frac{1}{p^r}\right)\frac{1}{p^{r-1}}
=\frac{p^r+p^{r-1}-1}{p^{2r-1}}.
\]
Equivalently, in the present induction one has $\alpha=1/p^2$, every noncentral centralizer has order $|G|/p$
and is abelian, so the preceding inequalities are all equalities.
\end{proof}

\begin{corollary}[Quantitative deficit from the extremal bound]\label{cor:deficit-alpha}
Let $G$ be finite non-abelian and let $p$ be the smallest prime dividing $|G|$.
Write $\alpha:=|\Z(G)|/|G|$ and set
\[
M_{r-1}(p):=\frac{p^{r-1}+p^{r-2}-1}{p^{2r-3}}.
\]
Then for every $r\ge 2$,
\[
P_r(G)\le \alpha\,M_{r-1}(p)+\frac{1-\alpha}{p^{r-1}}.
\]
In particular, since $\alpha\le 1/p^2$ for non-abelian $G$, one has the explicit deficit estimate
\[
\frac{p^r+p^{r-1}-1}{p^{2r-1}}-P_r(G)
\ \ge\
\left(\frac{1}{p^2}-\alpha\right)\frac{p^{r-1}-1}{p^{2r-3}}.
\]
\end{corollary}

\begin{proof}
Lemma~\ref{lem:onestep} gives $P_r(G)\le \alpha\,P_{r-1}(G)+(1-\alpha)p^{-(r-1)}$.
Apply Theorem~\ref{thm:Prbound} to bound $P_{r-1}(G)\le M_{r-1}(p)$, obtaining the first inequality.
The second inequality is the identity
\[
\Big(\frac{1}{p^2}M_{r-1}(p)+\Big(1-\frac{1}{p^2}\Big)\frac{1}{p^{r-1}}\Big)-\Big(\alpha\,M_{r-1}(p)+\frac{1-\alpha}{p^{r-1}}\Big)
=\left(\frac{1}{p^2}-\alpha\right)\left(M_{r-1}(p)-\frac{1}{p^{r-1}}\right),
\]
together with $M_{r-1}(p)-p^{-(r-1)}=(p^{r-1}-1)/p^{2r-3}$.
\end{proof}

\subsection{Specialization to commuting triples}
\begin{corollary}[Sharp bound for $P_3$]\label{cor:P3bound}
Let $G$ be a finite non-abelian group and $p$ the smallest prime dividing $|G|$.
Then
\[
P_3(G)\le \frac{p^3+p^2-1}{p^5}.
\]
In particular, if $|G|$ is even then $P_3(G)\le 11/32$.
Equality holds if and only if $G/\Z(G)\cong C_p\times C_p$.
\end{corollary}

\begin{proof}
This is Theorem~\ref{thm:Prbound} with $r=3$.
For even $|G|$, the smallest prime is $p=2$, giving $(8+4-1)/32=11/32$.
\end{proof}

\subsection{Specialization to commuting quadruples}
\begin{corollary}[Sharp bound for $P_4$]\label{cor:P4bound}
Let $G$ be a finite non-abelian group and let $p$ be the smallest prime dividing $|G|$.
Then
\[
P_4(G)\le \frac{p^4+p^3-1}{p^7}.
\]
In particular, if $|G|$ is even then $P_4(G)\le 23/128$.
Moreover, among all finite non-abelian groups the maximal possible value of $P_4(G)$ is $23/128$.
Equality holds if and only if $G/\Z(G)\cong C_p\times C_p$ (and in the global extremal case necessarily $p=2$).
\end{corollary}

\begin{proof}
This is Theorem~\ref{thm:Prbound} with $r=4$.
If $|G|$ is even then $p=2$ and $(2^4+2^3-1)/2^7=23/128$.
For the global maximum over non-abelian groups, note that for $p\ge 3$,
\[
\frac{p^4+p^3-1}{p^7}= \frac{1}{p^3}+\frac{1}{p^4}-\frac{1}{p^7}\le \frac{1}{27}+\frac{1}{81}<\frac{23}{128},
\]
so the largest value occurs at $p=2$.
\end{proof}

\begin{corollary}[Universal sharp constant and extremal isoclinism class]\label{cor:universal-bound}
Let $G$ be a finite non-abelian group. Then for every $r\ge 2$,
\[
P_r(G)\le c_r:=\frac{2^r+2^{r-1}-1}{2^{2r-1}}=\frac{3\cdot 2^{r-1}-1}{2^{2r-1}}.
\]
Equality holds if and only if $G/\Z(G)\cong C_2\times C_2$, i.e.\ if and only if $G$ lies in the isoclinism
class of $D_8$ (equivalently of $Q_8$).

In particular,
\[
c_2=\frac58,\qquad c_3=\frac{11}{32},\qquad c_4=\frac{23}{128},\qquad c_5=\frac{47}{512},\ \dots
\]
\end{corollary}

\begin{proof}
Let $p$ be the smallest prime dividing $|G|$.
By Theorem~\ref{thm:Prbound},
\[
P_r(G)\le \frac{p^r+p^{r-1}-1}{p^{2r-1}}.
\]
Since $p\ge 2$ and
\[
\frac{p^r+p^{r-1}-1}{p^{2r-1}}=p^{-(r-1)}+p^{-r}-p^{-(2r-1)}
\]
is strictly decreasing in $p$ for $p\ge 2$ (indeed, its derivative equals
$p^{-2r}\bigl(2r-1-(r-1)p^r-rp^{r-1}\bigr)<0$), we obtain
\[
P_r(G)\le \frac{2^r+2^{r-1}-1}{2^{2r-1}}=c_r.
\]
If equality holds, then necessarily $p=2$ and equality holds in Theorem~\ref{thm:Prbound}, forcing
$G/\Z(G)\cong C_2\times C_2$.
Conversely, if $G/\Z(G)\cong C_2\times C_2$, then Theorem~\ref{thm:Prbound} gives equality for all $r$.

Finally, assume $G/\Z(G)\cong C_2\times C_2$.
Then $G/\Z(G)$ is elementary abelian of rank $2$, so $G$ has class $2$ and $G'$ is generated by a single nontrivial commutator $[x,y]$.
For lifts $x,y$ of a basis of $G/\Z(G)$ one has $x^2,y^2\in \Z(G)$, and the class-$2$ identity
\[
[x^2,y]=[x,y]^2
\]
shows that every commutator has order dividing $2$.
Hence $G'\cong C_2$, and the commutator pairing
\[
G/\Z(G)\times G/\Z(G)\longrightarrow G'
\]
is the unique nondegenerate alternating bilinear form on a $2$-dimensional $\mathbb F_2$-vector space.
The same is true for both $D_8$ and $Q_8$, so $G$ is isoclinic to $D_8$ (equivalently to $Q_8$).
\end{proof}

\section{A stability gap near \texorpdfstring{$11/32$}{11/32}}

The value $11/32$ is classical in ordinary commuting-probability theory: Rusin classified finite groups with ordinary commuting probability greater than $11/32$ \cite{Rusin1979}. In the present higher-commutativity setting, the same number appears as the extremal value of $P_3$ for the family of non-abelian groups with
$G/\Z(G)\cong C_2\times C_2$ (e.g.\ $D_8$ and $Q_8$ and their central products with abelian groups).

The following proposition shows that, for commuting triples, to get \emph{close} to this value one must already be in the same extremal isoclinism family.

\begin{proposition}[Gap stability for commuting triples]\label{prop:gap}
Let $G$ be a finite non-abelian group and let $\alpha=|\Z(G)|/|G|$.
Then
\[
P_3(G)\le \frac{1}{4}+\frac{\alpha}{4}+\frac{\alpha^2}{2}.
\]
Consequently, if $P_3(G)>\frac{11}{36}$ then $|G:\Z(G)|=4$ and hence $G/\Z(G)\cong C_2\times C_2$.
\end{proposition}

\begin{proof}
We first show the bound.
Let $p$ be the smallest prime divisor of $|G|$.
By Lemma~\ref{lem:onestep},
\[
P_3(G)\le \alpha\,P_2(G)+\frac{1-\alpha}{p^2}
\le \alpha\,P_2(G)+\frac{1-\alpha}{4},
\]
since $p\ge 2$.
Also,
\[
P_2(G)=\frac{1}{|G|^2}\sum_{x\in G}|\C_G(x)|
\le \alpha+\frac{1}{|G|^2}\sum_{x\notin\Z(G)}\frac{|G|}{p}
=\alpha+\frac{1-\alpha}{p}
\le \alpha+\frac{1-\alpha}{2}
=\frac{1}{2}+\frac{\alpha}{2}.
\]
Substituting this into the previous inequality gives
\[
P_3(G)\le \alpha\Big(\frac12+\frac{\alpha}{2}\Big)+\frac{1-\alpha}{4}
=\frac14+\frac{\alpha}{4}+\frac{\alpha^2}{2}.
\]

Now suppose $P_3(G)>11/36$.
The function $f(\alpha)=\frac14+\frac{\alpha}{4}+\frac{\alpha^2}{2}$ is increasing for $\alpha\ge 0$,
and $f(1/6)=11/36$. Hence $P_3(G)>11/36$ forces $\alpha>1/6$, i.e.\ $|G:\Z(G)|<6$.
If $|G:\Z(G)|=1$ then $G$ is abelian, excluded; if $|G:\Z(G)|=2,3,5$ then $G/\Z(G)$ is cyclic so $G$
would be abelian. Therefore $|G:\Z(G)|=4$, and non-abelianity forces $G/\Z(G)\cong C_2\times C_2$.
\end{proof}

\begin{remark}[A general stability mechanism]\label{rem:general-stability}
The proof of Theorem~\ref{thm:Prbound} yields more than the sharp extremal constant: it gives a quantitative
deficit in terms of $\alpha=|\Z(G)|/|G|$; see Corollary~\ref{cor:deficit-alpha}.
Thus, for fixed $r$ and $p$, being within $\varepsilon$ of the extremal bound forces $\alpha$ to be within
$O(\varepsilon)$ of $1/p^2$.
The remaining step in a sharp stability theorem is to control further deficits coming from the distribution of
centralizer indices among noncentral elements.
\end{remark}

\begin{remark}
The numerical gap
\[
\frac{11}{32}-\frac{11}{36}=\frac{11}{288}\approx 0.03819
\]
is not claimed to be optimal; it comes from a short inequality that only uses
index-$2$ information about noncentral centralizers.
It nevertheless provides a clean \emph{rigidity window}: any non-abelian group whose
commuting-triple probability exceeds $11/36$ must already have central quotient of order $4$.
\end{remark}

\begin{proposition}[A $p$-group rigidity ladder for $P_r$]\label{prop:pgroup-ladder-Pr}
Let $p$ be a prime, let $G$ be a finite non-abelian $p$-group, and fix $r\ge 2$.
Write $|G:\Z(G)|=p^d$ with $d\ge 2$, so $\alpha:=|\Z(G)|/|G|=p^{-d}$.
Then
\begin{equation}\label{eq:pgroup-ladder}
P_r(G)\le U_{p,r}(d):=p^{-(r-1)}+(p-1)\sum_{j=1}^{r-1}p^{-(r+(d-1)j)}
=p^{-(r-1)}+(p-1)p^{-(r+d-1)}\frac{1-p^{-(d-1)(r-1)}}{1-p^{-(d-1)}}.
\end{equation}
Moreover, the sequence $d\mapsto U_{p,r}(d)$ is strictly decreasing for $d\ge 2$ and satisfies
$\lim_{d\to\infty}U_{p,r}(d)=p^{-(r-1)}$.
In particular, if $d\ge d_0$ then $P_r(G)\le U_{p,r}(d_0)$, and if $P_r(G)>U_{p,r}(d_0)$ then
$|G:\Z(G)|\le p^{d_0-1}$.
\end{proposition}

\begin{proof}
Let $\alpha:=|\Z(G)|/|G|=p^{-d}$.
Since $G$ is a $p$-group, $p$ is the smallest prime divisor of $|G|$.

For $k\ge 1$ define $f_k(\alpha)$ recursively by $f_1(\alpha)=1$ and, for $k\ge 2$,
\[
f_k(\alpha):=\alpha\,f_{k-1}(\alpha)+\frac{1-\alpha}{p^{k-1}}.
\]
Lemma~\ref{lem:onestep} implies by induction on $k$ that $P_k(G)\le f_k(\alpha)$ for all $k\ge 1$.
Unwinding the recursion gives
\[
f_k(\alpha)=\alpha^{k-1}+(1-\alpha)\sum_{j=0}^{k-2}\frac{\alpha^{j}}{p^{k-1-j}}
=\frac{1}{p^{k-1}}+\frac{p-1}{p}\alpha^{k-1}+(p-1)\sum_{j=1}^{k-2}\frac{\alpha^{j}}{p^{k-j}}.
\]
In particular, $f_k(\alpha)$ is a polynomial in $\alpha$ with nonnegative coefficients, hence is increasing on $[0,1]$.

Setting $k=r$ and $\alpha=p^{-d}$ and simplifying yields $f_r(p^{-d})=U_{p,r}(d)$ as in \eqref{eq:pgroup-ladder},
so $P_r(G)\le U_{p,r}(d)$.
The monotonicity in $d$ is immediate from the expression in \eqref{eq:pgroup-ladder}, since each summand is strictly decreasing in $d$.
Finally, the limit as $d\to\infty$ follows from \eqref{eq:pgroup-ladder} since the finite sum tends to $0$.
\end{proof}

\begin{corollary}[A $p$-group rigidity window for $P_r$]\label{cor:pgroup-gap-Pr}
Let $p$ be a prime, let $G$ be a finite non-abelian $p$-group, and fix $r\ge 2$.
Then either $G/\Z(G)\cong C_p\times C_p$, in which case
\[
P_r(G)=\frac{p^r+p^{r-1}-1}{p^{2r-1}},
\]
or else
\begin{equation}\label{eq:pgroup-gap-Pr}
P_r(G)\le B_{p,r}:=\frac{p^{2r-1}+p^{2r-3}-p^{2r-4}+\cdots+p-1}{p^{3r-2}}
=\frac{p^r+p^{r-1}-1}{p^{2r-1}}-\frac{(p^r-1)(p^{r-1}-1)}{(p+1)p^{3r-2}}.
\end{equation}
In particular, if $P_r(G)>B_{p,r}$ then necessarily $|G:\Z(G)|=p^2$ and $G/\Z(G)\cong C_p\times C_p$.
\end{corollary}

\begin{proof}
Write $|G:\Z(G)|=p^d$ with $d\ge 2$.
If $d=2$ then $G/\Z(G)$ has order $p^2$ and is not cyclic (otherwise $G$ would be abelian), hence
$G/\Z(G)\cong C_p\times C_p$, and Theorem~\ref{thm:Prbound} gives the stated value.
If $d\ge 3$ then Proposition~\ref{prop:pgroup-ladder-Pr} with $d_0=3$ gives
\[
P_r(G)\le U_{p,r}(3),
\]
and evaluating \eqref{eq:pgroup-ladder} at $d=3$ yields
\[
U_{p,r}(3)
=p^{-3(r-1)}+\Big(1-p^{-3}\Big)p^{-(r-1)}\sum_{j=0}^{r-2}p^{-2j}
=p^{-3(r-1)}+\Big(1-p^{-3}\Big)p^{-(r-1)}\frac{1-p^{-2(r-1)}}{1-p^{-2}},
\]
which simplifies to the explicit expressions in \eqref{eq:pgroup-gap-Pr}.
\end{proof}

\section{Class-\texorpdfstring{$2$ exponent-$p$}{2 exponent-p} groups and Heisenberg-type families}\label{sec:pgroup_orderp}

For class-$2$ exponent-$p$ $p$-groups, commutators are central and both $G/\Z(G)$ and $G'$ are naturally
$\mathbb F_p$-vector spaces, so the commutator induces an alternating bilinear map
$\,G/\Z(G)\times G/\Z(G)\to G'\,$ (equivalently, a linear map $\Lambda^2(G/\Z(G))\to G'$).
We begin with the extremal regime $|G'|=p$, where this map is symplectic and yields a clean recursion whose
output depends only on the single parameter $|G:\Z(G)|$.
We then record a general ``commutator tensor'' reduction valid for all class-$2$ exponent-$p$ groups and a
uniform $q$-symplectic recursion for Heisenberg-type groups over $\mathbb F_{p^m}$ (so $|G'|=p^m$), culminating in a
closed all-$r$ formula, explicit pole data, and rigidity inside that family.

\subsection{Symplectic reduction}

\begin{lemma}[Centrality and exponent-\texorpdfstring{$p$}{p} quotient]\label{lem:orderp-commutator-class2}
Let $G$ be a finite $p$-group with $|G'|=p$.
Then $G'\le \Z(G)$ and $G/\Z(G)$ is elementary abelian.
\end{lemma}

\begin{proof}
Conjugation induces a homomorphism $G\to \Aut(G')$.
Since $G$ is a $p$-group and $\Aut(G')\cong \Aut(C_p)$ has order $p-1$ coprime to $p$,
this homomorphism is trivial.
Thus $G'$ is central.

In a class-$2$ group one has $[x^p,y]=[x,y]^p$.
Here $[x,y]\in G'$ has order $p$, hence $[x^p,y]=1$ for all $y$, so $x^p\in \Z(G)$.
Therefore $(x\Z(G))^p=\Z(G)$ for every $x$, i.e.\ $G/\Z(G)$ has exponent $p$.
\end{proof}

\begin{lemma}[The commutator form]\label{lem:commutator-symplectic}
Let $G$ be as in Lemma~\ref{lem:orderp-commutator-class2}.
Fix an identification $G'\cong \mathbb F_p$.
Then the commutator induces a well-defined alternating bilinear form
\[
\beta: V\times V\to \mathbb F_p,
\qquad
V:=G/\Z(G),
\qquad
\beta(x\Z(G),y\Z(G))=[x,y].
\]
Moreover $\beta$ is nondegenerate, hence $\dim_{\mathbb F_p}V$ is even:
\(
|G:\Z(G)|=p^{2n}
\)
for some $n\ge 1$.
\end{lemma}

\begin{proof}
Since $G'$ is central, commutators are bilinear:
$[x_1x_2,y]=[x_1,y][x_2,y]$ and $[x,y_1y_2]=[x,y_1][x,y_2]$.
Thus $\beta$ is bilinear over $\mathbb F_p$ (Lemma~\ref{lem:orderp-commutator-class2} gives that $V$ is an $\mathbb F_p$-vector space).
It is alternating because $[x,x]=1$.

If $v=x\Z(G)\in V$ lies in the radical, then $[x,y]=1$ for all $y\in G$, hence $x\in \Z(G)$ and $v=0$.
Thus $\beta$ is nondegenerate, so $V$ is a symplectic space and $\dim V=2n$ for some $n$.
\end{proof}

\subsection{A recursion and explicit formulas}

Fix $n\ge 0$ and let $V_n=\mathbb F_p^{2n}$ equipped with any nondegenerate alternating form
$\langle\cdot,\cdot\rangle$ (unique up to change of basis).
For $r\ge 1$ define
\[
I_r(n):=\big|\{(v_1,\dots,v_r)\in V_n^r:\ \langle v_i,v_j\rangle=0\ \forall i<j\}\big|,
\qquad
P_r^{(p)}(n):=\frac{I_r(n)}{p^{2nr}}.
\]
Thus $P_r^{(p)}(n)$ is the probability that $r$ random vectors in $V_n$ are pairwise orthogonal.

\begin{theorem}[Symplectic recursion for \texorpdfstring{$|G'|=p$}{|G'|=p}]\label{thm:symplectic-recursion}
Let $G$ be a finite $p$-group with $|G'|=p$ and write $|G:\Z(G)|=p^{2n}$.
Then for every $r\ge 1$,
\[
P_r(G)=P_r^{(p)}(n),
\]
so $P_r(G)$ depends only on $p$ and $n$ (and not on $|\Z(G)|$).

Moreover, for every $n\ge 1$ and $r\ge 2$ one has the recursion
\[
I_r(n)=I_{r-1}(n)+(p^{2n}-1)p^{r-1}I_{r-1}(n-1),
\]
equivalently
\[
P_r^{(p)}(n)=p^{-2n}P_{r-1}^{(p)}(n)+p^{1-r}\bigl(1-p^{-2n}\bigr)P_{r-1}^{(p)}(n-1),
\]
with initial conditions $P_r^{(p)}(0)=1$ and $P_1^{(p)}(n)=1$.
\end{theorem}

\begin{proof}
By Lemma~\ref{lem:commutator-symplectic}, commuting in $G$ is detected on $V=G/\Z(G)$ by the symplectic form:
$(x_1,\dots,x_r)$ commutes pairwise if and only if $\beta(x_i\Z(G),x_j\Z(G))=0$ for all $i<j$.
Each $v\in V$ has exactly $|\Z(G)|$ lifts to $G$, so
\[
|\Comm_r(G)|=|\Z(G)|^r\,I_r(n).
\]
Since $|G|=|\Z(G)|\,p^{2n}$, dividing by $|G|^r$ gives $P_r(G)=I_r(n)/p^{2nr}=P_r^{(p)}(n)$.

For the recursion, fix $n\ge 1$ and count $I_r(n)$ by the first coordinate $v_1$.
If $v_1=0$, then $(v_2,\dots,v_r)$ is any orthogonal $(r-1)$-tuple in $V_n$, giving $I_{r-1}(n)$ choices.
If $v_1\neq 0$, then each $v_i$ must lie in the orthogonal hyperplane $v_1^\perp$, whose quotient
$v_1^\perp/\langle v_1\rangle$ is a symplectic space of dimension $2n-2$.
Every orthogonal $(r-1)$-tuple in $v_1^\perp/\langle v_1\rangle$ has exactly $p^{r-1}$ lifts to $v_1^\perp$,
and there are $p^{2n}-1$ choices for $v_1\neq 0$.
This gives the stated formula for $I_r(n)$.
The recursion for $P_r^{(p)}(n)$ is obtained by dividing by $p^{2nr}$.
\end{proof}

\begin{corollary}[Closed formulas for $P_2$ and $P_3$]\label{cor:orderp-P2P3}
Let $G$ be a finite $p$-group with $|G'|=p$ and $|G:\Z(G)|=p^{2n}$.
Then
\[
P_2(G)=\frac{1}{p}+\left(1-\frac{1}{p}\right)p^{-2n},
\qquad
P_3(G)=\frac{p^{2n}+p^3-1}{p^{2n+3}}.
\]
Equivalently,
\[
\kappa_2(G)=|G|^2P_3(G)=|G|^2\,\frac{p^{2n}+p^3-1}{p^{2n+3}}.
\]
\end{corollary}

\begin{proof}
Compute $P_2^{(p)}(n)$ and $P_3^{(p)}(n)$ from Theorem~\ref{thm:symplectic-recursion} using $P_1^{(p)}(n)=1$
and $P_r^{(p)}(0)=1$.
The formula for $\kappa_2(G)$ is Theorem~\ref{thm:burnside} for $r=2$.
\end{proof}

\begin{theorem}[Rank-one commutator: commuting probabilities determine isoclinism]\label{thm:rank1-isoclinism}
Let $p$ be a prime and let $G,H$ be finite $p$-groups with $|G'|=|H'|=p$.
Write $|G:\Z(G)|=p^{2n}$ and $|H:\Z(H)|=p^{2m}$.
Then the following are equivalent:
\begin{enumerate}
\item $G$ and $H$ are isoclinic;
\item $n=m$ (equivalently $|G:\Z(G)|=|H:\Z(H)|$);
\item $P_2(G)=P_2(H)$;
\item $P_3(G)=P_3(H)$.
\end{enumerate}
In this case $P_r(G)=P_r(H)$ for all $r\ge 2$.

Moreover, the parameter $n$ (and hence the isoclinism class) is recovered from $P_2(G)$ and $P_3(G)$ by
\[
p^{2n}=\frac{p-1}{pP_2(G)-1}=\frac{p^3-1}{p^3P_3(G)-1}.
\]
\end{theorem}

\begin{proof}
$(1)\Rightarrow(2)$ is immediate since isoclinism gives an isomorphism $G/\Z(G)\cong H/\Z(H)$.

$(2)\Rightarrow(1)$: by Lemma~\ref{lem:commutator-symplectic}, fixing identifications $G'\cong\mathbb F_p$ and
$H'\cong\mathbb F_p$, the commutators define nondegenerate alternating forms
$\beta_G$ on $V_G:=G/\Z(G)$ and $\beta_H$ on $V_H:=H/\Z(H)$.
If $n=m$ then $V_G$ and $V_H$ have the same $\mathbb F_p$-dimension $2n$, and any two nondegenerate alternating forms
on a $2n$-dimensional $\mathbb F_p$-vector space are isometric.
Thus there exists an isomorphism $\phi:V_G\to V_H$ and an isomorphism $\psi:G'\to H'$ such that
$\psi\circ \beta_G=\beta_H\circ(\phi\times\phi)$, which is exactly the isoclinism condition
(Definition~\ref{def:isoclinism}).

$(2)\Leftrightarrow(3)$ and the first displayed formula follow from Corollary~\ref{cor:orderp-P2P3}.
Similarly, $(2)\Leftrightarrow(4)$ and the second formula follow by solving the $P_3$-identity in
Corollary~\ref{cor:orderp-P2P3} for $p^{2n}$.
Finally, if $n=m$ then Theorem~\ref{thm:symplectic-recursion} gives $P_r(G)=P_r^{(p)}(n)=P_r^{(p)}(m)=P_r(H)$ for all $r\ge 2$.
\end{proof}

\begin{corollary}[A closed formula for $P_4$]\label{cor:orderp-P4}
Under the hypotheses of Corollary~\ref{cor:orderp-P2P3},
\[
P_4(G)
=\frac{1}{p^6}+\frac{p^5+p^3-p^2-1}{p^6}\,p^{-2n}
+\frac{p^4-p^3-p+1}{p^4}\,p^{-4n}.
\]
In particular, if $n=1$ then $P_4(G)=(p^4+p^3-1)/p^7$, i.e.\ the extremal value from Corollary~\ref{cor:P4bound}.
\end{corollary}

\begin{proof}
Apply Theorem~\ref{thm:symplectic-recursion} with $r=4$ and substitute the explicit expression for $P_3(G)$
from Corollary~\ref{cor:orderp-P2P3}.
A short simplification yields the stated polynomial in $p^{-2n}$.
The specialization $n=1$ follows by direct simplification (or by Corollary~\ref{cor:P4bound}).
\end{proof}

\begin{remark}[Large symplectic rank limit]\label{rem:symplectic-limit}
For fixed $r$ and $p$, letting $n\to\infty$ in the recursion of Theorem~\ref{thm:symplectic-recursion} shows that
\[
\lim_{n\to\infty} P_r^{(p)}(n)=p^{-\binom{r}{2}}.
\]
\end{remark}

\subsection{Beyond \texorpdfstring{$|G'|=p$}{|G'|=p}}\label{subsec:commutator-tensor}

The symplectic reduction above uses the strongest possible uniformity: a $1$-dimensional commutator space and a
nondegenerate alternating form.
More generally, many $p$-groups (including all groups of nilpotency class $2$ and exponent $p$) admit a linear-algebra
model in which the commutator is an alternating bilinear map with values in an $\mathbb F_p$-vector space.
This viewpoint turns higher commuting probabilities into \emph{isotropic tuple counts} for that bilinear map.

\begin{definition}[Class-$2$ exponent-$p$ groups and the commutator map]\label{def:class2-exp-p}
Let $p$ be a prime and let $G$ be a finite $p$-group of nilpotency class $2$ and exponent $p$.
(If $p=2$ then exponent $p$ forces $G$ abelian, so the non-abelian case is relevant only for odd $p$.)
Set
\[
V:=G/\Z(G),\qquad W:=G'.
\]
Then $V$ and $W$ are naturally $\mathbb F_p$-vector spaces, and the commutator induces an alternating $\mathbb F_p$-bilinear map
\[
\beta:V\times V\to W,\qquad \beta(x\Z(G),y\Z(G))=[x,y],
\]
equivalently a linear map of $\mathbb F_p$-vector spaces
\[
\widetilde\beta:\Lambda^2 V\to W.
\]
We call $\widetilde\beta$ (or $\beta$) the \emph{commutator tensor} of $G$.
\end{definition}

\begin{proposition}[Reduction to isotropic tuples]\label{prop:class2-reduction}
Let $G$ be a finite $p$-group of class $2$ and exponent $p$, with commutator tensor $\widetilde\beta:\Lambda^2V\to W$
as in Definition~\ref{def:class2-exp-p}.
For $r\ge 1$ set
\[
N_r(V,\beta):=\big|\{(v_1,\dots,v_r)\in V^r:\ \beta(v_i,v_j)=0\ \forall i<j\}\big|.
\]
Then
\[
|\Comm_r(G)|=|\Z(G)|^r\,N_r(V,\beta),
\qquad\text{and hence}\qquad
P_r(G)=\frac{N_r(V,\beta)}{|V|^r}.
\]
In particular, $P_r(G)$ depends only on the alternating bilinear map $\beta$ on $V$ (and not on $|\Z(G)|$).
\end{proposition}

\begin{proof}
Since $G$ has class $2$, commutators are central, and if $z\in \Z(G)$ then $[xz,y]=[x,y]$ and $[x,yz]=[x,y]$.
Thus whether an ordered tuple $(x_1,\dots,x_r)\in G^r$ is pairwise commuting depends only on the cosets
$(x_1\Z(G),\dots,x_r\Z(G))\in V^r$ and is equivalent to the isotropy conditions $\beta(v_i,v_j)=0$.

Each $r$-tuple $(v_1,\dots,v_r)\in V^r$ has exactly $|\Z(G)|^r$ lifts to $G^r$.
Therefore $|\Comm_r(G)|=|\Z(G)|^r\,N_r(V,\beta)$.
Dividing by $|G|^r=(|\Z(G)|\,|V|)^r$ gives $P_r(G)=N_r(V,\beta)/|V|^r$.
\end{proof}

\begin{proposition}[A rank-distribution formula for $P_2$]\label{prop:P2-rankdistribution}
Let $G$ be a finite $p$-group of class $2$ and exponent $p$ with commutator map $\beta:V\times V\to W$ as in
Definition~\ref{def:class2-exp-p}.
For $v\in V$ let $\beta_v:V\to W$ be the linear map $\beta_v(w):=\beta(v,w)$.
Then
\[
P_2(G)=\frac{1}{|V|}\sum_{v\in V} p^{-\mathrm{rk}(\beta_v)}
=\mathbf E_{v\in V}\big[p^{-\mathrm{rk}(\beta_v)}\big].
\]
In particular, if $\mathrm{rk}(\beta_v)=\dim W$ for all $v\neq 0$, then
\[
P_2(G)=p^{-\dim V}+\left(1-p^{-\dim V}\right)p^{-\dim W}.
\]
\end{proposition}

\begin{proof}
For fixed $v\in V$, the number of $w\in V$ with $\beta(v,w)=0$ is $|\ker \beta_v|=p^{\dim V-\mathrm{rk}(\beta_v)}$.
Thus
\[
N_2(V,\beta)=\sum_{v\in V} |\ker \beta_v|
=\sum_{v\in V} p^{\dim V-\mathrm{rk}(\beta_v)}.
\]
Divide by $|V|^2=p^{2\dim V}$ and use Proposition~\ref{prop:class2-reduction}.
The final identity is the special case where all nonzero $v$ have the same rank.
\end{proof}

\begin{proposition}[Isotropic-span recursion]\label{prop:isotropic-span-recursion}
Let $\beta:V\times V\to W$ be any alternating bilinear map between finite $\mathbb F_p$-vector spaces.
For a subspace $U\le V$ define its orthogonal complement
\[
U^\perp:=\{v\in V:\ \beta(u,v)=0\ \forall u\in U\}.
\]
Then for every $r\ge 1$,
\[
N_{r+1}(V,\beta)=\sum_{(v_1,\dots,v_r)\ \text{\rm isotropic}} \big|\langle v_1,\dots,v_r\rangle^\perp\big|,
\]
where $\langle v_1,\dots,v_r\rangle$ denotes the linear span.
Equivalently,
\[
P_{r+1}(G)=P_r(G)\cdot \mathbf E\!\left[\frac{|\langle v_1,\dots,v_r\rangle^\perp|}{|V|}\ \Big|\ (v_1,\dots,v_r)\ \text{\rm isotropic}\right]
\]
for every class-$2$ exponent-$p$ group $G$ with commutator map $\beta$.
\end{proposition}

\begin{proof}
Fix an isotropic $r$-tuple $(v_1,\dots,v_r)$.
An element $v_{r+1}\in V$ extends it to an isotropic $(r+1)$-tuple if and only if
$\beta(v_i,v_{r+1})=0$ for all $i$, i.e.\ if and only if $v_{r+1}\in \langle v_1,\dots,v_r\rangle^\perp$.
Summing over isotropic $r$-tuples gives the first identity.
The second identity is obtained by dividing by $|V|^{r+1}$ and rewriting the sum as a conditional expectation.
\end{proof}

\subsection{Heisenberg-type groups over finite fields}\label{subsec:heisenberg}

A particularly uniform higher-rank family arises when the commutator tensor is symplectic over a finite field extension.

\begin{definition}[$\mathbb F_{p^m}$-Heisenberg type]\label{def:heisenberg-type}
Let $p$ be an odd prime and let $q:=p^m$.
We say that a finite $p$-group $G$ is of \emph{$\mathbb F_q$-Heisenberg type of rank $n$} if:
\begin{enumerate}
\item $G$ has nilpotency class $2$ and exponent $p$;
\item $G'=\Z(G)$ is elementary abelian of order $q$ (so $G'\cong \mathbb F_q$ as additive groups);
\item $V:=G/\Z(G)$ carries the structure of an $\mathbb F_q$-vector space of dimension $2n$ such that,
under an identification $G'\cong \mathbb F_q$, the commutator map
\(
\beta:V\times V\to \mathbb F_q
\)
is a nondegenerate alternating $\mathbb F_q$-bilinear form.
\end{enumerate}
Equivalently, at the level of commutator geometry, $G$ is represented by a nondegenerate alternating $\mathbb F_q$-bilinear form on $\mathbb F_q^{2n}$; this determines the associated isoclinism data, not a preferred group isomorphism type.
\end{definition}

For such a group, commuting in $G$ is detected on $V$ by the $\mathbb F_q$-symplectic form.
Define, for $r\ge 1$,
\[
I_r^{(q)}(n):=\big|\{(v_1,\dots,v_r)\in (\mathbb F_q^{2n})^r:\ \langle v_i,v_j\rangle=0\ \forall i<j\}\big|,
\qquad
P_r^{(q)}(n):=\frac{I_r^{(q)}(n)}{q^{2nr}},
\]
where $\langle\cdot,\cdot\rangle$ is any nondegenerate alternating form on $\mathbb F_q^{2n}$.

\begin{theorem}[$q$-symplectic recursion]\label{thm:qsymplectic-recursion}
Let $G$ be of $\mathbb F_q$-Heisenberg type of rank $n$ (Definition~\ref{def:heisenberg-type}).
Then for every $r\ge 1$,
\[
P_r(G)=P_r^{(q)}(n),
\]
so $P_r(G)$ depends only on $q=p^m$ and $n$.

Moreover, for every $n\ge 1$ and $r\ge 2$ one has the recursion
\[
I_r^{(q)}(n)=I_{r-1}^{(q)}(n)+(q^{2n}-1)\,q^{r-1}\,I_{r-1}^{(q)}(n-1),
\]
equivalently
\[
P_r^{(q)}(n)=q^{-2n}P_{r-1}^{(q)}(n)+q^{1-r}\bigl(1-q^{-2n}\bigr)P_{r-1}^{(q)}(n-1),
\]
with initial conditions $P_r^{(q)}(0)=1$ and $P_1^{(q)}(n)=1$.
\end{theorem}

\begin{proof}
By Proposition~\ref{prop:class2-reduction}, $P_r(G)$ equals the probability that $r$ random vectors in $V\cong\mathbb F_q^{2n}$
are pairwise orthogonal for the symplectic form, which is exactly $P_r^{(q)}(n)$.

The recursion is the same counting argument as in Theorem~\ref{thm:symplectic-recursion}, now over $\mathbb F_q$.
If $v_1=0$ there are $I_{r-1}^{(q)}(n)$ choices for $(v_2,\dots,v_r)$.
If $v_1\neq 0$, there are $q^{2n}-1$ choices for $v_1$, and the remaining vectors must lie in $v_1^\perp$.
The quotient $v_1^\perp/\langle v_1\rangle$ is a symplectic $\mathbb F_q$-space of dimension $2n-2$,
and each tuple in the quotient has $q^{r-1}$ lifts to $v_1^\perp$ (adding independent multiples of $v_1$).
This gives the stated formula for $I_r^{(q)}(n)$.
Dividing by $q^{2nr}$ yields the recursion for $P_r^{(q)}(n)$.
\end{proof}

\begin{theorem}[Closed hierarchy and rigidity in the $\mathbb F_q$-Heisenberg family]\label{thm:heisenberg-full-hierarchy}\label{thm:rank2-heisenberg-isoclinism}
Let $p$ be an odd prime, let $q:=p^m$, and let $G$ be of $\mathbb F_q$-Heisenberg type of rank $n$
in the sense of Definition~\ref{def:heisenberg-type}. Write
\[
V:=G/\Z(G)\cong \mathbb F_q^{2n}.
\]
For $0\le k\le n$, let $L_{n,k}(q)$ denote the number of $k$-dimensional totally isotropic
$\mathbb F_q$-subspaces of $V$. Then
\[
L_{n,k}(q)=\prod_{i=0}^{k-1}\frac{q^{2n-2i}-1}{q^{k-i}-1}.
\]

For every $r\ge 1$ one has
\[
I_r^{(q)}(n)
=
\sum_{k=0}^{\min(n,r)}
L_{n,k}(q)\prod_{i=0}^{k-1}(q^r-q^i),
\]
and hence
\[
|\Comm_r(G)|
=
|\Z(G)|^r
\sum_{k=0}^{\min(n,r)}
L_{n,k}(q)\prod_{i=0}^{k-1}(q^r-q^i),
\]
equivalently
\[
P_r(G)
=
q^{-2nr}
\sum_{k=0}^{\min(n,r)}
L_{n,k}(q)\prod_{i=0}^{k-1}(q^r-q^i).
\]
In particular, the full hierarchy $\{P_r(G)\}_{r\ge2}$ depends only on $(q,n)$.

Now let $q':=p^{m'}$, and let $H$ be another finite $p$-group of $\mathbb F_{q'}$-Heisenberg type of rank $n'$.
Then the following are equivalent:
\begin{enumerate}
\item $G$ and $H$ are isoclinic;
\item $(q,n)=(q',n')$;
\item $P_2(G)=P_2(H)$ and $P_3(G)=P_3(H)$;
\item $P_r(G)=P_r(H)$ for all $r\ge2$.
\end{enumerate}
Moreover, in (3) the parameter $q$ is the unique positive root of
\[
\bigl(P_2(G)-P_3(G)\bigr)X^2+\bigl(P_2(G)-1\bigr)X+\bigl(P_2(G)-1\bigr)=0,
\]
and then
\[
q^{-2n}=\frac{qP_2(G)-1}{q-1}.
\]

Finally, with the manuscript convention
\[
\mathcal P_G(z):=\sum_{r\ge2}P_r(G)z^{r-2},
\]
the series $\mathcal P_G(z)$ is rational and its pole set is contained in
\[
\{q^n,q^{n+1},\dots,q^{2n}\}.
\]
In particular, the pole at $z=q^n$ is present.
\end{theorem}

\begin{proof}
Let $\langle\cdot,\cdot\rangle$ be the nondegenerate alternating $\mathbb F_q$-bilinear form on
$V\cong \mathbb F_q^{2n}$ coming from the commutator map.

First we compute $L_{n,k}(q)$.
Let $\mathcal B_{n,k}(q)$ be the set of ordered $k$-tuples
\[
(u_1,\dots,u_k)\in V^k
\]
such that $u_1,\dots,u_k$ are linearly independent and span a totally isotropic subspace.
We count $\mathcal B_{n,k}(q)$ in two ways.

Choose such a tuple inductively.
After choosing $u_1,\dots,u_{j-1}$, their span
\[
U_{j-1}:=\langle u_1,\dots,u_{j-1}\rangle
\]
is totally isotropic of dimension $j-1$, hence $U_{j-1}\subseteq U_{j-1}^{\perp}$ and
\[
\dim_{\mathbb F_q}(U_{j-1}^{\perp})=2n-(j-1).
\]
To keep the enlarged span totally isotropic and linearly independent, the next vector must lie in
$U_{j-1}^{\perp}\setminus U_{j-1}$. Therefore the number of choices for $u_j$ is
\[
|U_{j-1}^{\perp}|-|U_{j-1}|=q^{2n-j+1}-q^{j-1}.
\]
Thus
\[
|\mathcal B_{n,k}(q)|
=
\prod_{j=1}^{k}\bigl(q^{2n-j+1}-q^{j-1}\bigr).
\]

On the other hand, each $k$-dimensional totally isotropic subspace has exactly
\[
\prod_{j=1}^{k}(q^k-q^{j-1})
\]
ordered bases. Hence
\[
L_{n,k}(q)
=
\frac{\prod_{j=1}^{k}(q^{2n-j+1}-q^{j-1})}
     {\prod_{j=1}^{k}(q^k-q^{j-1})}.
\]
Factoring out $q^{j-1}$ from numerator and denominator in each factor gives
\[
L_{n,k}(q)
=
\prod_{j=1}^{k}\frac{q^{2n-2j+2}-1}{q^{k-j+1}-1}
=
\prod_{i=0}^{k-1}\frac{q^{2n-2i}-1}{q^{k-i}-1},
\]
as claimed.

Now let
\[
N_r(V):=\bigl|\{(v_1,\dots,v_r)\in V^r:\ \langle v_i,v_j\rangle=0\ \text{for all }i<j\}\bigr|.
\]
By definition, $N_r(V)=I_r^{(q)}(n)$.
We count $N_r(V)$ by the dimension of the span of the tuple.

If $(v_1,\dots,v_r)$ is pairwise orthogonal, then
\[
U:=\langle v_1,\dots,v_r\rangle
\]
is a totally isotropic subspace of $V$.
Let $k:=\dim U$. Then necessarily $0\le k\le \min(n,r)$.
Conversely, fix a $k$-dimensional totally isotropic subspace $U\le V$.
The pairwise orthogonal $r$-tuples whose span is exactly $U$ are precisely the ordered $r$-tuples in $U^r$
that generate $U$.

Choose an $\mathbb F_q$-linear isomorphism $U\cong \mathbb F_q^k$.
Then ordered generating $r$-tuples of $U$ are in bijection with surjective linear maps
\[
T:\mathbb F_q^r\twoheadrightarrow \mathbb F_q^k,
\qquad
T(e_i)=v_i.
\]
The number of such surjections equals the number of rank-$k$ $k\times r$ matrices over $\mathbb F_q$.
Transposing, this is the number of injective linear maps
\[
\mathbb F_q^k\hookrightarrow \mathbb F_q^r,
\]
which is
\[
\prod_{i=0}^{k-1}(q^r-q^i):
\]
choose the image of the first basis vector in $q^r-1$ ways, the second outside its span in $q^r-q$ ways,
and so on.

Therefore, for each fixed $k$-dimensional totally isotropic subspace $U$, the number of pairwise orthogonal
$r$-tuples spanning $U$ is
\[
\prod_{i=0}^{k-1}(q^r-q^i).
\]
Summing over all such $U$ gives
\[
N_r(V)
=
\sum_{k=0}^{\min(n,r)}
L_{n,k}(q)\prod_{i=0}^{k-1}(q^r-q^i).
\]
Since $N_r(V)=I_r^{(q)}(n)$, this proves the formula for $I_r^{(q)}(n)$.

Now apply Proposition~\ref{prop:class2-reduction}. Because $|V|=q^{2n}$,
\[
|\Comm_r(G)|=|\Z(G)|^r\,N_r(V)
\]
and
\[
P_r(G)=\frac{N_r(V)}{|V|^r}=q^{-2nr}N_r(V),
\]
which yields the displayed formulas for $|\Comm_r(G)|$ and $P_r(G)$.
In particular, $P_r(G)$ depends only on $(q,n)$.

We now prove the rigidity statement.

Assume first that $G$ and $H$ are isoclinic.
Then $G'\cong H'$ as groups, so
\[
|G'|=|H'|,
\]
hence $q=q'$.
Also
\[
G/\Z(G)\cong H/\Z(H),
\]
so
\[
q^{2n}=|G:\Z(G)|=|H:\Z(H)|=(q')^{2n'}=q^{2n'}.
\]
Therefore $n=n'$, proving $(1)\Rightarrow(2)$.

Conversely, assume $(q,n)=(q',n')$.
Choose identifications
\[
G'\cong \mathbb F_q,\qquad H'\cong \mathbb F_q.
\]
Then $V_G:=G/\Z(G)$ and $V_H:=H/\Z(H)$ are both $2n$-dimensional symplectic spaces over $\mathbb F_q$.
Choose symplectic bases
\[
(e_1,\dots,e_n,f_1,\dots,f_n)\quad\text{for }V_G,
\]
\[
(e_1',\dots,e_n',f_1',\dots,f_n')\quad\text{for }V_H.
\]
The unique $\mathbb F_q$-linear map $\phi:V_G\to V_H$ sending $e_i\mapsto e_i'$ and $f_i\mapsto f_i'$
is an isometry of alternating forms.
Let $\psi:G'\to H'$ be the chosen identification.
Then
\[
\psi\circ \beta_G=\beta_H\circ(\phi\times\phi),
\]
so $G$ and $H$ are isoclinic. Thus $(2)\Rightarrow(1)$.

The implication $(2)\Rightarrow(4)$ is immediate from the closed formula for $P_r$,
and $(4)\Rightarrow(3)$ is trivial.

It remains to prove $(3)\Rightarrow(2)$.
Taking $r=2$ and $r=3$ in the closed formula yields
\[
P_2(G)=\frac1q+\Bigl(1-\frac1q\Bigr)q^{-2n},
\qquad
P_3(G)=\frac1{q^3}+\Bigl(1-\frac1{q^3}\Bigr)q^{-2n}.
\]
Set
\[
a:=q^{-2n}.
\]
Then
\[
P_2(G)=q^{-1}+(1-q^{-1})a,\qquad
P_3(G)=q^{-3}+(1-q^{-3})a.
\]
Subtracting and comparing with $1-P_2(G)$ gives
\[
\frac{P_2(G)-P_3(G)}{1-P_2(G)}
=
\frac{q^{-1}-q^{-3}}{1-q^{-1}}
=
q^{-1}+q^{-2}.
\]
Thus $q$ is determined uniquely by $P_2(G)$ and $P_3(G)$, because the function
$x\mapsto x^{-1}+x^{-2}$ is strictly decreasing on $(0,\infty)$.
Equivalently, eliminating $a$ yields the quadratic
\[
\bigl(P_2(G)-P_3(G)\bigr)X^2+\bigl(P_2(G)-1\bigr)X+\bigl(P_2(G)-1\bigr)=0,
\]
whose unique positive root is $X=q$.

Once $q$ is known, the formula for $P_2(G)$ gives
\[
a=q^{-2n}=\frac{qP_2(G)-1}{q-1}.
\]
Since the map $n\mapsto q^{-2n}$ is injective on $\mathbb Z_{\ge0}$, this determines $n$.
Hence equality of $P_2$ and $P_3$ for $G$ and $H$ forces $(q,n)=(q',n')$, proving $(3)\Rightarrow(2)$.
This completes the equivalence of (1)--(4).

Finally, for each fixed $k$ the factor
\[
\prod_{i=0}^{k-1}(q^r-q^i)
\]
is a polynomial in $q^r$ of degree $k$.
Therefore there exist coefficients $c_0,\dots,c_n\in \mathbb Q(q)$, depending only on $(q,n)$, such that
\[
P_r(G)=\sum_{j=0}^{n} c_j\,q^{-(2n-j)r}.
\]
Hence
\[
\mathcal P_G(z)
=
\sum_{r\ge2}P_r(G)z^{r-2}
=
\sum_{j=0}^{n} c_j\sum_{r\ge2}q^{-(2n-j)r}z^{r-2}
=
\sum_{j=0}^{n}\frac{c_j\,q^{-2(2n-j)}}{1-z/q^{\,2n-j}}.
\]
Thus $\mathcal P_G(z)$ is rational and its poles lie among
\[
q^n,q^{n+1},\dots,q^{2n}.
\]
Moreover, the term $k=n$ contributes
\[
L_{n,n}(q)\,q^{-2nr}\cdot q^{nr}=L_{n,n}(q)\,q^{-nr},
\]
and no term with $k<n$ produces $q^{-nr}$.
So the coefficient of $q^{-nr}$ is $L_{n,n}(q)>0$, and the pole at $z=q^n$ is present.
\end{proof}

\begin{corollary}[Low-rank specializations]\label{cor:heisenberg-low-rank}\label{cor:heisenberg-P2P3}\label{cor:heisenberg-P4}
Under the hypotheses of Theorem~\ref{thm:heisenberg-full-hierarchy},
\[
P_2(G)=\frac{1}{q}+\left(1-\frac{1}{q}\right)q^{-2n},
\qquad
P_3(G)=\frac{q^{2n}+q^3-1}{q^{2n+3}},
\]
and
\[
P_4(G)
=\frac{1}{q^6}+\frac{q^5+q^3-q^2-1}{q^6}\,q^{-2n}
+\frac{q^4-q^3-q+1}{q^4}\,q^{-4n}.
\]
\end{corollary}

\begin{proof}
Substitute $r=2,3,4$ into the closed formula of Theorem~\ref{thm:heisenberg-full-hierarchy} and simplify.
\end{proof}

\begin{corollary}[Maximum-order abelian data in the Heisenberg family]\label{cor:heisenberg-maxabelian}
Under the hypotheses of Theorem~\ref{thm:heisenberg-full-hierarchy},
\[
m(G)=q^{n+1}
\qquad\text{and}\qquad
N_{\max}(G)=L_{n,n}(q)=\prod_{i=1}^{n}(q^i+1).
\]
\end{corollary}

\begin{proof}
Let $A\le G$ be abelian. Since $G$ has class $2$, the subgroup $A\Z(G)$ is abelian and contains $A$.
Thus, when bounding the order of abelian subgroups, and in particular when counting maximum-order abelian subgroups, we may assume without loss of generality that $\Z(G)\le A$. Indeed, if $A$ has maximum possible order, then $A\Z(G)$ is abelian and $|A\Z(G)|\ge |A|$, so maximality of the order forces $A=A\Z(G)$.

Now set $W:=A/\Z(G)\le V:=G/\Z(G)$. Then $W$ is an $\mathbb F_p$-subspace of $V$ with $\beta(W,W)=0$.
Its $\mathbb F_q$-span $\mathbb F_qW$ is again totally isotropic, because $\beta$ is $\mathbb F_q$-bilinear.
Hence
\[
\dim_{\mathbb F_q}(\mathbb F_qW)\le n,
\]
so
\[
|W|=p^{\dim_{\mathbb F_p}W}\le p^{m\dim_{\mathbb F_q}(\mathbb F_qW)}\le p^{mn}=q^n.
\]
Therefore
\[
|A|=|\Z(G)|\,|W|\le q\cdot q^n=q^{n+1}.
\]
Equality holds exactly when $W$ is a Lagrangian $\mathbb F_q$-subspace of $V$, and distinct Lagrangians give distinct
abelian subgroups containing $\Z(G)$. Thus $m(G)=q^{n+1}$ and $N_{\max}(G)=L_{n,n}(q)$.
\end{proof}

\begin{remark}[Large rank limit]\label{rem:qsymplectic-limit}
For fixed $r$ and $q$, letting $n\to\infty$ in the recursion of Theorem~\ref{thm:qsymplectic-recursion} yields
\[
\lim_{n\to\infty} P_r^{(q)}(n)=q^{-\binom{r}{2}}.
\]
\end{remark}

\subsection{Rank two beyond Heisenberg type}\label{subsec:rank2-ses}

The $\mathbb F_{p^2}$-Heisenberg-type family gives one clean rank-two model where the commutator tensor is essentially
a symplectic form over a field extension.  However, already when $|G'|=p^2$ there are many other class-$2$ exponent-$p$
$p$-groups with ``maximally large'' commutator fibres.  A standard and widely studied condition is that every noncentral
element has full commutator image:
\begin{equation}\label{eq:full-contraction-rank}
x\notin \Z(G)\quad\Longrightarrow\quad [x,G]=G'.
\end{equation}
In the language of bilinear commutator maps, \eqref{eq:full-contraction-rank} says that every nonzero contraction has full
rank (``full contraction rank'').  Groups satisfying \eqref{eq:full-contraction-rank} are precisely the Camina $p$-groups of
nilpotency class $2$, equivalently the \emph{semi-extraspecial} $p$-groups; see Lewis's proof-bearing treatment of generalized Camina groups \cite{Lewis2009Camina} and his published survey \cite[Theorems~5.1--5.2]{Lewis2019Survey}.

\begin{proposition}[Full contraction rank forces uniform centralizers and determines $P_2$]\label{prop:fullcontraction-P2}
Let $p$ be a prime and let $G$ be a finite $p$-group of nilpotency class $2$ and exponent $p$.
Write $V:=G/\Z(G)$ and $W:=G'$, with $|V|=p^d$ and $|W|=p^m$.
Assume $G$ has full contraction rank in the sense of \eqref{eq:full-contraction-rank}.
Then every noncentral element has centralizer index $|W|$:
\[
x\notin \Z(G)\quad\Longrightarrow\quad |G:\C_G(x)|=|G'|=p^m,
\]
and the commuting probability satisfies the explicit formula
\begin{equation}\label{eq:P2-fullcontraction}
P_2(G)=p^{-d}+\bigl(1-p^{-d}\bigr)p^{-m}
=\frac{|\Z(G)|}{|G|}+\Bigl(1-\frac{|\Z(G)|}{|G|}\Bigr)\frac{1}{|G'|}.
\end{equation}
In particular, in the rank-two case $|G'|=p^2$, one has $P_2(G)=p^{-d}+(1-p^{-d})p^{-2}$.
\end{proposition}

\begin{proof}
Let $\beta:V\times V\to W$ be the alternating commutator map.
For $x\in G$ with image $v\in V$, the centralizer quotient satisfies
\[
\C_G(x)/\Z(G)\ \cong\ \{u\in V:\beta(v,u)=0\}=\ker(\beta_v),
\]
where $\beta_v:V\to W$ is the contraction $u\mapsto\beta(v,u)$.
If $x\notin\Z(G)$ then $v\neq0$, so by \eqref{eq:full-contraction-rank} the map $\beta_v$ is surjective.
Hence $|\ker(\beta_v)|=|V|/|W|=p^{d-m}$, i.e.\ $|\C_G(x)|=|\Z(G)|\,p^{d-m}=|G|/|W|$, proving the displayed index formula.

Finally,
\[
P_2(G)=\frac{1}{|G|^2}\sum_{x\in G}|\C_G(x)|
=\frac{|\Z(G)|\cdot |G|+\bigl(|G|-|\Z(G)|\bigr)\cdot |G|/|W|}{|G|^2}
=\frac{|\Z(G)|}{|G|}+\Bigl(1-\frac{|\Z(G)|}{|G|}\Bigr)\frac{1}{|W|},
\]
which is \eqref{eq:P2-fullcontraction}.
\end{proof}

\begin{theorem}[Small-order rank-two rigidity in the semi-extraspecial family (literature)]\label{thm:ses-rank2-smallorders}
Let $p$ be a prime.  Among semi-extraspecial $p$-groups with $|G'|=p^2$ (equivalently, Camina $p$-groups of class $2$),
there is a \emph{unique} isoclinism class of groups of order $p^6$, and also a unique isoclinism class of groups of order $p^8$.
\end{theorem}

\begin{proof}
For $|G|=p^6$ this follows from a slight extension of Verardi's classification, as recorded in
\cite[\S1]{LewisMaglione2018}; see also \cite[Lemma~5.19]{Yadav2005}.
For $|G|=p^8$ this is Theorem~1.1 of \cite{LewisMaglione2018}.
\end{proof}

\begin{remark}[Beyond $p^8$: many isoclinism classes]\label{rem:ses-many}
The uniqueness in Theorem~\ref{thm:ses-rank2-smallorders} breaks quickly: already at order $p^{10}$ there are
$p+3-\gcd(2,p)$ isoclinism classes of semi-extraspecial groups with derived subgroup of order $p^2$
\cite[Theorem~1.2]{LewisMaglione2018}, and the number grows rapidly with the order \cite[Theorem~1.3]{LewisMaglione2018}.
This illustrates that rank-two commutator tensors admit substantial moduli beyond the genus-$1$ (Heisenberg-type) situation
captured in Theorem~\ref{thm:heisenberg-full-hierarchy}.
\end{remark}

\section{Conclusion}

The hierarchy $\{P_r(G)\}_{r\ge2}$ has a rigid structural side. At the low-rank end, the exact class-number formulas for $P_3(G)$ and $P_4(G)$ recover the untwisted Drinfeld-double and quantum-triple simple counts, understood here purely at the level of simple counts. Exact prime-index formulas then force strong restrictions near the top of the hierarchy, the known universal bound isolates the extremal isoclinism class $G/Z(G)\cong C_p\times C_p$, and the stability gap near $11/32$ shows that commuting triples retain meaningful quantitative rigidity away from the top value.

In class-$2$ exponent-$p$ groups, the symplectic reduction turns higher commutativity into a problem about isotropic tuples. This produces closed formulas and classification consequences ranging from the rank-one commutator case to the full $\mathbb F_q$-Heisenberg family, where the pair $(P_2(G),P_3(G))$ already determines the isoclinism class. A natural next step is to push these rigidity phenomena beyond Heisenberg type, especially in rank two and semi-extraspecial settings, and to understand how much of the full hierarchy survives under broader isoclinism-preserving deformations.

\end{document}